\documentclass{amsart}

\newcommand{\qmod}[1]{\operatorname{QM}[#1]}
\newcommand{\qm}{\operatorname{QM} }

\newcommand{\xdt}[1]{x_{\Delta_{#1}}}
\newcommand{\udt}[1]{u_{\Delta_{#1}}}
\newcommand{\ydt}[1]{y_{\Delta_{#1}}}

\usepackage{amsmath}
\usepackage{amsthm}
\usepackage{mathtools}
\usepackage{graphicx}
\usepackage{enumerate}
\usepackage{multirow}
\usepackage{color}
\usepackage{amsfonts}
\usepackage{amssymb}
\usepackage{mathrsfs}
\usepackage{amscd}
\usepackage{url}
\usepackage{xcolor}

\newcommand{\re}{\mathbb{R}}

\newcommand{\N}{\mathbb{N}}

\newcommand{\U}{\mathbb{U}}

\newcommand{\diag}{\mbox{diag}}

\newcommand{\lmd}{\lambda}
\newcommand{\Lmd}{\Lambda}

\newcommand{\eps}{\epsilon}

\newcommand{\Dt}{\Delta}
\def\af{\alpha}

\def\gm{\gamma}
\def\rank{\mbox{rank}}

\newcommand{\sig}{\sigma}
\newcommand{\Sig}{\Sigma}

\newcommand{\st}{\mbox{s.t.}}

\newcommand{\reff}[1]{(\ref{#1})}

\newcommand{\tr}[1]{\mbox{trace}(#1)}

\newcommand{\mc}[1]{\mathcal{#1}}

\newcommand{\ddd}{,\ldots,}
\newcommand{\lip}{\left<}
\newcommand{\rip}{\right>}

\newcommand{\bdes}{\begin{description}}
	\newcommand{\edes}{\end{description}}

\newcommand{\bal}{\begin{align}}
	\newcommand{\eal}{\end{align}}

\newcommand{\bnum}{\begin{enumerate}}
	\newcommand{\enum}{\end{enumerate}}

\newcommand{\bit}{\begin{itemize}}
	\newcommand{\eit}{\end{itemize}}

\newcommand{\bea}{\begin{eqnarray}}
	\newcommand{\eea}{\end{eqnarray}}
\newcommand{\be}{\begin{equation}}
	\newcommand{\ee}{\end{equation}}

\newcommand{\baray}{\begin{array}}
	\newcommand{\earay}{\end{array}}

\newcommand{\bsry}{\begin{subarray}}
	\newcommand{\esry}{\end{subarray}}

\newcommand{\bca}{\begin{cases}}
	\newcommand{\eca}{\end{cases}}

\newcommand{\bcen}{\begin{center}}
	\newcommand{\ecen}{\end{center}}

\newcommand{\bbm}{\begin{bmatrix}}
	\newcommand{\ebm}{\end{bmatrix}}

\newcommand{\bmx}{\begin{matrix}}
	\newcommand{\emx}{\end{matrix}}

\newcommand{\bpm}{\begin{pmatrix}}
	\newcommand{\epm}{\end{pmatrix}}

\newcommand{\btab}{\begin{tabular}}
	\newcommand{\etab}{\end{tabular}}

\newtheorem{theorem}{Theorem}[section]

\theoremstyle{definition}
\newtheorem{example}[theorem]{Example}
\newtheorem{exm}[theorem]{Example}
\newtheorem{alg}[theorem]{Algorithm}

\newtheorem{remark}[theorem]{Remark}

\numberwithin{equation}{section}

\begin{document}

\title[Sparse Polynomial Optimization with Matrix Constraints]
{Sparse Polynomial Optimization with Matrix Constraints}

\author[Jiawang~Nie]{Jiawang Nie}
\author[Zheng~Qu]{Zheng Qu}
\author[Xindong~Tang]{Xindong Tang}
\author[Linghao~Zhang]{Linghao Zhang}

\address{Jiawang Nie and Linghao Zhang,
	Department of Mathematics, University of California San Diego,
	9500 Gilman Drive, La Jolla, CA, USA, 92093.}
\email{njw@math.ucsd.edu,liz010@ucsd.edu}

\address{Zheng Qu, 
	School of Mathematical Sciences, 
	Shenzhen University, 
	518061, P.R. China.}
\email{zhengqu@szu.edu.cn}

\address{Xindong Tang, 
	Department of Mathematics,
	Hong Kong Baptist University,
	Kowloon Tong, Kowloon, Hong Kong.}
\email{xdtang@hkbu.edu.hk}

\date{}



\begin{abstract}
This paper studies the hierarchy of sparse matrix Moment-SOS relaxations for solving sparse polynomial optimization problems with matrix constraints.
First, we prove a sufficient and necessary condition for the sparse hierarchy to be tight.
Second, we discuss how to detect the tightness and extract minimizers.
Third, for the convex case,
we show that the hierarchy of the sparse matrix Moment-SOS relaxations is tight, under some general assumptions.
In particular, we show that the sparse matrix Moment-SOS relaxation is tight for every order when the problem is SOS-convex.
Numerical experiments are provided to show the efficiency of the sparse relaxations.
\end{abstract}

\maketitle

\section{Introduction}

Let $x \coloneqq (x_1, \ldots, x_n)$ be an $n$-dimensional vector of variables,
and $\Dt_1, \ldots, \Dt_m$ be subsets of $[n] \coloneqq \{1,\ldots, n\}$
such that $\Dt_1\cup\cdots \cup \Dt_m = [n]$.
For each $\Dt_i = \{j_1, \ldots, j_{n_i} \}$, denote the subvector
$\xdt{i}  \coloneqq  (x_{j_1}, \ldots, x_{j_{n_i}} ).$
We consider the sparse matrix polynomial optimization problem
\be \label{spar:matPOP}
\left\{ \baray{cl}
\displaystyle \min_{x\in\re^n} & f(x)  \coloneqq f_1(\xdt{1}) + \cdots + f_m(\xdt{m}) \\
\st & G_i(\xdt{i}) \succeq 0  ,\, i = 1, \ldots, m.
\earay \right.
\ee
In the above, each $f_i$ is a polynomial in $\xdt{i}$ and
each $G_i$ is a symmetric matrix polynomial in $\xdt{i}$.
We denote by $f_{\min}$ the minimum value of \reff{spar:matPOP} and let
\be\label{K_di}
K_{\Dt_i} \coloneqq \{\xdt{i} \in \re^{n_i}: G_i(\xdt{i}) \succeq 0 \}.
\ee
The feasible set of \reff{spar:matPOP} is
\[
K = \bigcap\limits_{i=1}^{m} \{x \in \re^n : G_i(\xdt{i}) \succeq 0 \}.
\]

\begin{example}\label{ex:exp_sol}
Consider the following matrix polynomial optimization:
\be\label{eq:explicit}
\left\{
\begin{array}{cl}
	\displaystyle \min_{x\in \re^3} & -x_1x_2 + (x_3-x_2)^2 \\
	\st & G_i(\xdt{i})\succeq 0,\ i=1,2.
\end{array}\right.
\ee
In the above, $\Dt_1 = \{1,2\}$, $\Dt_2 = \{2,3\}$, $f_1 = -x_1x_2$, $f_2 = (x_3-x_2)^2$, and
\[\begin{aligned}
	G_1(x_1,x_2)\,\coloneqq\,\left[\begin{array}{cc}
		x_1 & x_1x_2 \\ x_1x_2 & x_2^2
	\end{array}
	\right],\quad
	G_2(x_2,x_3)\,\coloneqq\,\left[\begin{array}{ccc}
		x_2+x_3 & x_2 & 0 \\ x_2 & x_2 & 0\\
		0 & 0 & 1-x_2
	\end{array}
	\right].
\end{aligned}\]
One can check that $f_{\min} = -1$ and the minimizer is $x^* = (1,1,1)$.
This is because if $ x = (x_1,x_2,x_3)$ is feasible, then either $x_2 = 0$, $x_1 \ge 0$, $x_3 \ge 0$, or 
$0 < x_2\le 1$, $0\le x_1 \le 1$, $x_3 \ge 0$,
which implies that $f(x) \ge -1$ for all feasible $x$.
\end{example}

Matrix constrained polynomial optimization problems can be solved by the {\it dense} matrix Moment-SOS hierarchy of semidefinite relaxations, which are introduced in \cite{henrion2006convergent,Hol04}.
Denote the matrix set $\mc{G} \coloneqq \{1, G_1 \ddd G_m\}$. The quadratic module of polynomials generated by $\mc{G}$ is
\[
\qm[\mc{G}] \coloneqq \Big\{ \sum_{j=1}^{s} P_j^T B_j P_j : B_j \in \mc{G}, s \in \N, P_j \in \re[x]^{\text{len}(B_j)} \Big\}.
\]
In the above, len($B_j$) denotes the length of $B_j$.
For an integer $k$, the degree-$2k$ truncation $\qm[\mc{G}]_{2k}$ is the set of all polynomials that can be represented as above, with $2\deg(P_j)+\deg(B_j) \le 2k$ for every $j$.
The $k$th order dense SOS relaxation of \reff{spar:matPOP} is
\be\label{dense_sos}
\left\{ \baray{cl}
\max & \gm \\
\st & f - \gm \in   \qmod{\mc{G}}_{2k}.
\earay \right.
\ee
Its dual optimization problem is the $k$th order dense moment relaxation:
\be\label{dense_mom}
\left\{ \baray{cl}
\min & \langle f, y \rangle \\
\st &  L_{G_i}^{(k)}[y] \succeq 0 ,\, i = 1, \ldots, m,   \\
&  M_k[y]\succeq 0,  y_0 = 1, \, \\
& y \in \re^{\N^n_{2k} } .
\earay \right.
\ee
We refer to Section~\ref{ssc:denmom} for the above notation.
For $k=1,2, \ldots$, the sequence of \reff{dense_sos}-\reff{dense_mom} is called the dense matrix Moment-SOS  hierarchy. It produces a sequence of lower bounds for the minimum value $f_{\min}$ of \reff{spar:matPOP}, which converges to $f_{\min}$ under the archimedean condition \cite{henrion2006convergent,Hol04}.
This matrix Moment-SOS hierarchy is said to be {\it tight} if the optimal value \reff{dense_sos} equals $f_{\min}$ for some $k$.
It is shown in \cite{Huang2024} that this hierarchy is tight under some optimality conditions.
We refer to \cite{Guo2024robust,Huang2024complexity,Huang2024,nie2011PMI,SchererHol06}
for related work about matrix constrained polynomial optimization.

When the number of variables or the relaxation order increases,
the sizes of the dense matrix Moment-SOS relaxations \reff{dense_sos}-\reff{dense_mom} grow rapidly.
It is expensive to solve (\ref{spar:matPOP}) on a large scale.
Thus, it is important to exploit sparsity to
improve computational efficiency.
In this paper, we focus on the matrix polynomial optimization problem
with the sparsity pattern as in (\ref{spar:matPOP}).
This is referred to as {\it correlative sparsity} in some literature,
to be distinguished from {\it term sparsity}  (\cite{TSSOS,ChTS}).
Sparse polynomial optimization has wide applications.
We refer to \cite{grimm2007note,huang2024sparse, Klep2022, Lasserre06,nie2009sparse,NQZT24,Qu2024,WangMag21,CSTSSOS,waki2006sums} for related work on sparsity.
Moreover, we refer to \cite{Kim2011,Zheng2023chordal} for representations of sparse matrix polynomials.
In Section~\ref{sc:control}, we give an application of sparse matrix polynomial optimization in the multisystem static $\mc{H}_2$ controller synthesis.

In this paper, we study the {\it sparse} matrix Moment-SOS hierarchy of semidefinite relaxations for solving \reff{spar:matPOP}.
For a given degree $k$,
the $k$th order sparse SOS relaxation is
\be  \label{ksos}
\left\{ \baray{cl}
\max & \gm \\
\st & f - \gm \in   \qmod{G}_{spa, 2k}.
\earay \right.
\ee
Its dual optimization problem is the $k$th order sparse moment relaxation:
\be  \label{kmom}
\left\{ \baray{cl}
\min & \langle f, y \rangle \coloneqq
\langle f_1, \ydt{1} \rangle + \cdots + \langle f_m, \ydt{m} \rangle  \\
\st &  L_{G_i}^{(k)}[\ydt{i}] \succeq 0 ,\, i = 1, \ldots, m,   \\
&  M_{\Dt_i}^{(k)}[\ydt{i}]\succeq 0,\, i = 1, \ldots, m,   \\
&  y_0 = 1, \,  y \in \re^{ \U_{k} } .
\earay \right.
\ee
The optimal values of \reff{ksos} and  \reff{kmom}
are denoted as $f_k^{spa}$ and $f_k^{smo}$ respectively.
The symbol $\ydt{i}$ denotes the subvector of $y$
that is labelled by monomial powers in $\xdt{i}$.
We refer to Section~\ref{sc:pre} for the notation in the above.
As we increase the relaxation order $k$, the sequence of relaxation problems
(\ref{ksos})-(\ref{kmom}) gives the {\it sparse matrix Moment-SOS hierarchy}
for solving (\ref{spar:matPOP}). We have convergence
$f^{spa}_k\to f_{\min}$ when $\Dt_1\ddd \Dt_m$ satisfy the
{\it running intersection property} (RIP) (see \cite{Lasserre06}) and every
$\qm_{\Dt_i}[G_i]$ is {\it archimedean} \cite{Kojima09}.
Compared with the dense relaxations \reff{dense_sos}-\reff{dense_mom}, the sparse version \reff{ksos}-\reff{kmom} has
positive semidefinite (psd) matrix constraints or variables with much smaller sizes.

For the special case where every matrix $G_i$ is diagonal,
(\ref{spar:matPOP}) reduces to the scalar constrained sparse polynomial optimization problem, 
and the sparse relaxations \reff{ksos}-\reff{kmom}  reduce to the sparse Moment-SOS relaxation hierarchy studied in~\cite{Lasserre06,nie2009sparse,waki2006sums}.

In the recent work \cite{NQZT24},
it is shown that the sparse Moment-SOS hierarchy (for the scalar case) is tight
if and only if the objective function can be written as a sum of sparse nonnegative polynomials,
each of which belongs to the corresponding sparse quadratic module.
However, there is very little work on the tightness of the sparse matrix Moment-SOS relaxations.

\subsection*{Contribution}

This paper investigates conditions for the tightness of the sparse matrix Moment-SOS hierarchy (\ref{ksos})-(\ref{kmom}).
Our major results are:

\begin{itemize}
	\item We prove a sufficient and necessary condition for the tightness of the sparse matrix Moment-SOS hierarchy of (\ref{ksos})-(\ref{kmom}).
	Specifically,
	we show that $f-f_{\min} \in \qmod{G}_{spa, 2k}$ (this implies the relaxation (\ref{ksos}) is tight) if and only if there exist sparse polynomials $p_i \in \re[\xdt{i}]_{2k}$ such that
	\[
	\boxed{
		\begin{gathered}
			p_1  + \cdots + p_m  + f_{\min}  =0, \\
			f_i + p_i \in \qm_{\Dt_i}[G_i]_{2k}, \ i=1,\ldots, m.
		\end{gathered}
	}
	\]
	The above means that $f - f_{\min}$ can be equivalently expressed as a sum of sparse nonnegative polynomials, each of which belongs to the corresponding sparse quadratic module,
	i.e.,
	\[
	f- f_{\min} = (f_1 + p_1) + \cdots + (f_m + p_m).
	\]
	
	\item We give explicit conditions for the tightness of the sparse hierarchy of (\ref{ksos})-(\ref{kmom}) when the sparse matrix polynomial optimization is convex.
	In particular, we show that if the objective and constraining matrix polynomials are SOS-convex, then the moment relaxation (\ref{kmom}) is tight for all relaxation orders.
	
	\item We show that under certain conditions, the tightness of sparse matrix Moment-SOS hierarchy can be detected by the flat truncation, and minimizers can be extracted from moment matrices.
\end{itemize}

This paper is organized as follows. Some basics on matrix polynomial optimization and algebraic
geometry are reviewed in Section~\ref{sc:pre}.
Section~\ref{sc:char} gives a characterization for the tightness of the sparse matrix Moment-SOS hierarchy.
In Section~\ref{sc:flat}, we study the flat truncation for certifying tightness of moment relaxations.
Section~\ref{sc:convex} gives some sufficient conditions for the tightness when the sparse matrix polynomial optimization is convex.
Some numerical experiments are presented in Section~\ref{sc:ne}.

\section{Preliminaries}
\label{sc:pre}

\subsection*{Notation}

Denote by $\re$ (resp., $\N$) the set of real numbers (resp., nonnegative integers).
For a positive integer $k$, let $[k] \coloneqq \{1\ddd k\}$. For a real number $t$, $\lfloor t\rfloor$ (resp., $\lceil t \rceil$) denotes the largest integer that is smaller than or equal to (resp., the smallest integer that is larger than or equal to) $t$.
For a positive integer $n$, $\re^n$ (resp., $\N^n$) stands for the set of $n$-dimensional vectors whose entries are real numbers (resp., nonnegative integers).
For a matrix $X$, $X^T$ denotes the transpose of $X$.
For $u,v\in \re^n$, $\langle u, v\rangle \coloneqq v^Tu$.
The Euclidean norm of $u$ is $\|u\| \coloneqq \sqrt{u^Tu}$.
For a positive integer $\ell$,
denote by $\mc{S}^{\ell}$ the set of all $\ell$-by-$\ell$ real symmetric matrices. For $X\in \mc{S}^{\ell}$, $X \succeq 0$ (resp., $X \succ 0$) means $X$ is positive semidefinite (resp., positive definite), and we denote by $\mc{S}^{\ell}_+$ the set of all $\ell$-by-$\ell$ positive semidefinite matrices. For $X,Y\in \mc{S}^{\ell}$,
$
\langle X , Y \rangle \coloneqq  \tr{XY}
$
and  $X\succeq Y$ means $X-Y\succeq 0$.

For $x:=(x_1,\ldots,x_n)$ and $\af = (\af_1, \ldots, \af_n) \in \N^n$,
denote $x^\alpha \coloneqq x_1^{\af_1} \cdots x_n^{\af_n}$.
The ring of polynomials with real coefficients in $x$  is denoted as  $\re[x]$.
Denote by  $\re[x]^{\ell}$ (resp., $\re[x]^{\ell_1\times \ell_2}$)
the set of $\ell$-dimensional real polynomial vectors
(resp. $\ell_1$-by-$\ell_2$ real matrix polynomials) in $x$.
Denote by $\mc{S}\re[x]^{\ell \times \ell}$ the set of $\ell$-by-$\ell$
symmetric matrix polynomials in $x$ with real coefficients.
For $p\in \re[x]$, $\deg(p)$ denotes the degree of $p$.
For a degree $k$, $\re[x]_k$ denotes the subset of polynomials in $\re[x]$ with degrees at most $k$.
For $P\in \re[x]^{\ell_1\times \ell_2}$, let
\[
\deg(P):=\max \left\{\deg(P_{ij}): i\in [\ell_1],j\in [\ell_2]\right\}.
\]
For $f \in \re[x]$, $\nabla f(x)$ denotes the gradient of $f$ at the point $x$
and $\nabla^2 f(x)$ denotes the Hessian of $f$ at $x$.
For the matrix polynomial $G\in \mc{S}\re[x]^{\ell \times \ell}$,
the derivative of $G$ at a point $x$ is the linear mapping $\nabla G(x) \colon \re^n \to \mc{S}^{\ell}$ such that
\be\label{eq:deri_G}
d\coloneqq(d_1\ddd d_n) \mapsto \nabla G(x)[d] \, \coloneqq \, \sum_{i=1}^n d_i  \nabla_{x_i} G(x).
\ee
In the above,
$
\nabla_{x_i} G(x) \, \coloneqq \, \frac{\partial G(x)}{\partial x_i}.
$
The adjoint $\nabla G(x)^*$ is the linear mapping from $\mc{S}^{\ell}$
to $\re^{n}$ such that for $X \in \mc{S}^{\ell}$,
\be\label{eq:adj_deri_G}
\nabla G(x)^*[X]  = \left[ \lip \nabla_{x_1} G(x),X\rip \ \ldots\ \lip \nabla_{x_n} G(x),X\rip \right]^T.
\ee
At a point $u\in\re^n$ with $G(u)\succeq 0$, the {\it nondegeneracy condition} (NDC) holds for $G$ at $u$ if
\be\label{eq:NDC} \mbox{Im}\,\nabla G(u) + T  = \mc{S}^{\ell},\ee
where $\mbox{Im}\,\nabla G(u)$ is the image of the linear map $G(u)$, and
\[
T\,\coloneqq\,\{ X\in \mc{S}^{\ell} \colon v^T X v = 0 \ \forall\,v\in\ker \, G(u) \}.
\]
Here, $\ker \, G(u)$ denotes the kernel of $G(u)$, i.e., the null space of $G(u)$.
We refer to~\cite{Shapiro97,Sun06} for more details about nonlinear semidefinite programs.

For a subset $\Dt_i \subseteq [n]$, denote by $\re^{\Dt_i}$ the space of real vectors in the form of $\xdt{i}$.
The ring of polynomials in $\xdt{i}$ with real coefficients is denoted as $\re[\xdt{i}]$.
The notation $\re[\xdt{i}]^k$, $\re[\xdt{i}]^{k_1\times k_2}$ and $\mc{S}\re[\xdt{i}]^{\ell \times \ell}$ are similarly defined.
For $f\in \re[x]$ (resp., $G\in \mc{S}\re[x]^{\ell \times \ell}$), $\nabla_{\xdt{i}} f$
(resp., $\nabla_{\xdt{i}}G$) denotes the vector of partial derivatives of $f$ (resp., $G$) with respect to variables in $\xdt{i}$.
The Hessian $\nabla^2_{\xdt{i}} f$ is similarly defined.

\subsection{SOS polynomials and quadratic modules}
\label{ssc:sospoly}
A polynomial $\sigma \in \re[x]$ is said to be a {\it sum of squares} (SOS)
if there exist polynomials $p_1, \ldots, p_s \in \re[x]$ such that
$\sigma = p_1^2 + \cdots + p_s^2$.
The cone of SOS polynomials in $x$ is denoted as $\Sigma[x]$, and
\[
\Sig[x]_{2k} \coloneqq \Sig[x] \cap \re[x]_{2k}.
\]
The cone of $t$-by-$t$ SOS matrix polynomials in $x$ is
\[
\Sig[x]^{t \times t} \coloneqq \Big\{  P^T P:   P \in \re[x]^{s\times t} \mathrm{~for~some~} s\in \N \Big\}.
\]
For each $i\in [m]$,  $\Sigma [x_{\Delta_i}]$ denotes the set of SOS polynomials in $\xdt{i}$.
The truncation $\Sigma [x_{\Delta_i}]_{2k}$ and the cone of SOS matrix polynomials $\Sigma [x_{\Dt_i}]^{t\times t}$ in $\xdt{i}$ are similarly defined.

For $G_i\in \mc{S}\re[x_{\Dt_i}]^{\ell_i \times \ell_i}$, its {\it quadratic modules} in
$\mc{S}\re[x_{\Dt_i}]^{t \times t}$ and $\mc{S}\re[x]^{t \times t}$ are respectively:
\[
\qm_{\Dt_i}[G_i]^{t\times t} \coloneqq \Sig[x_{\Dt_i}]^{t \times t} + \Big\{\sum_{j=1}^{s} P^T_jG_iP_j : s \in \N, P_j \in \re[x_{\Dt_i}]^{ \ell_i\times t} \Big\},
\]
\[
\qm[G_i]^{t\times t} \coloneqq \Sig[x]^{t \times t} + \Big\{\sum_{j=1}^{s} P^T_jG_iP_j : s \in \N, P_j \in \re[x]^{ \ell_i\times t} \Big\}.
\]
When $t=1$, we denote
\[
\qm[G_i]:=\qm[G_i]^{1\times 1}, \quad \qm_{\Dt_i}[G_i]:=\qm_{\Dt_i}[G_i]^{1\times 1}.
\]
The quadratic module $\qm_{\Dt_i}[G_i]$ is said to be {\it archimedean} if there exists a scalar $R>0$ such that
$ R-\|x_{\Delta_i}\|^2 \in \qm_{\Dt_i}[G_i]$.
For an even degree $2k$,  denote the $2k$-truncation of $\qm_{\Dt_i}[G_i]$:
\[
\qm_{\Dt_i}[G_i]_{2k} \coloneqq \Sig[x_{\Dt_i}]_{2k} + \Big\{\sum_{j=1}^{s} P_j^TG_iP_j \Big\vert
\begin{array}{l}
s \in \N, P_j\in \re[x_{\Dt_i}]^{\ell_i}
\\ 2\deg(P_j) + \deg(G_i) \leq 2k
\end{array}
\Big \}.
\]
The truncation  $\qm[G_i]_{2k}$ is defined similarly.
For a matrix polynomial tuple $G \coloneqq (G_1 \ddd G_m)$ such that each $G_i\in\mc{S}\re[x_{\Dt_i}]^{\ell_i \times \ell_i}$, we denote
\be\label{eq:qm_spa}
\left\{
\begin{array}{rcl}
\qm[G]_{spa} & \coloneqq & \qm_{\Dt_1}[G_1] + \cdots + \qm_{\Dt_m}[G_m],\\
\qm[G]_{spa,2k} & \coloneqq & \qm_{\Dt_1}[G_1]_{2k} + \cdots + \qm_{\Dt_m}[G_m]_{2k}.
\end{array}
\right.
\ee

\subsection{Dense moments}
\label{ssc:denmom}

For a power vector $\alpha \coloneqq (\alpha_1, \ldots, \alpha_n) \in \mathbb{N}^n$, denote $|\alpha| \coloneqq \alpha_1 + \cdots + \alpha_n$.
The notation
\[
\mathbb{N}_d^n \coloneqq \{\alpha \in \mathbb{N}^n : |\alpha| \leq d\}
\]
stands for the set of monomial powers with degrees at most $d$.
The symbol $\mathbb{R}^{\mathbb{N}_d^n}$ denotes the space of all real vectors labeled by $\alpha \in \mathbb{N}_d^n$.
A vector $y \coloneqq \left(y_{\alpha}\right)_{\alpha \in \mathbb{N}^n_{2k}}$ is called a \textit{truncated multi-sequence} (tms) of degree $2k$.
For $y \in \mathbb{R}^{\mathbb{N}^n_{2k}}$, the \textit{Riesz functional} determined by $y$ is the linear functional $\mathscr{L}_y$ acting on $\re[x]_{2k}$ such that
\[
\mathscr{L}_y\Big( \sum\limits_{\alpha \in \mathbb{N}^n_{2k}} p_{\alpha}x^{\alpha} \Big)
\coloneqq \sum\limits_{\alpha \in \mathbb{N}^n_{2k}} p_{\alpha}y_{\alpha}.
\]
For convenience, we denote
\[
\langle p, y \rangle \coloneqq \mathscr{L}_y(p), \quad p\in \re[x]_{2k}.
\]
The \textit{localizing matrix} of $p$ generated by $y$ is
\[
L_p^{(k)}[y] \coloneqq \mathscr{L}_y(p(x) \cdot [x]_{s_1}[x]_{s_1}^T).
\]
In the above, the linear operator is applied entry-wise and
\[
s_1 \coloneqq \lfloor k - \deg(p)/2 \rfloor, \quad [x]_{s_1} \coloneqq (x^{\alpha})_{\alpha\in \N^n_{s_1}}.
\]
In particular, for $p=1$, we get the \textit{moment matrix} $M_k[y] \coloneqq L_1^{(k)}[y]$.
More details for this can be found in \cite{nie2023moment}.
For a matrix polynomial $F \in \mc{S}\re[x]^{\ell \times \ell}$ with entries as
$
F \, \coloneqq  \,
\big(  F_{st} \big)_{1 \le s,t \le \ell},
$
its localizing matrix is the $\ell \times \ell$ block matrix 
(let $s_2 \coloneqq \lfloor k - \deg(F)/2 \rfloor$)
\[
L_{F}^{(k)}[y] \, \coloneqq  \,
\big( L_{F_{st}}^{({k})}[y] \big)_{1 \le s,t \le \ell},
\]
where $L_{F_{st}}^{({k})}[y] \, \coloneqq  \,
\mathscr{L}_y(F_{st}(x) \cdot [x]_{s_2}[x]_{s_2}^T)$.

\subsection{Sparse moments}
\label{ssc:sparmom}

For each $\xdt{i}$, denote the set of monomial powers
\[
\N^{\Dt_i}\coloneqq \{\af = (\alpha_1\ddd \alpha_{n})\in\N^n :
\alpha_j = 0\,\, \forall j\notin \Dt_{i}\}.
\]
For a degree $d$, denote
$
\N_{d}^{\Dt_i} \coloneqq \{\af\in \N^{\Dt_i}:   |\af| \le d \}.
$
The vector of all monomials in $\xdt{i}$ with degrees up to $d$,
listed in the graded lexicographic order, is denoted as 
\be \label{[xDti]d}
[\xdt{i}]_d \, = \,  \big( x^\af \big)_{ \af \in \N_{d}^{\Dt_i}   }.
\ee
Denote by $\re^{\N_d^{\Dt_i}}$
the space of all real vectors whose entries are labeled by $\af \in \mathbb{N}_d^{\Dt_i}$.
A vector $y_{\Dt_i} \in \re^{\N_{d}^{\Dt_i}}$
is called a truncated multi-sequence (tms) of degree $d$.
The Riesz functional determined by $y_{\Dt_i}$
is the linear functional $\mathscr{L}_{y_{\Dt_i}}$ acting on $\re[\xdt{i}]$ such that
\[
\mathscr{L}_{y_{\Dt_i}}(x^{\alpha}) \, = \,
(y_{\Dt_i} )_{\alpha} \quad \mbox{for each} \quad \alpha\in \mathbb{N}_d^{\Dt_i} .
\]
This induces the bilinear operation $\langle \cdot, \cdot \rangle:
\re[x_{\Dt_i}]_d \times \re^{\mathbb{N}_d^{\Dt_i}}\rightarrow \re$ such that
\be \label{<p,yDt>}
\langle p, y_{\Dt_i} \rangle \coloneqq \mathscr{L}_{y_{\Dt_i}}(p).
\ee
When $d=2k$, the  {\it localizing matrix} of $p\in\re[\xdt{i}]_{2k}$,
generated by $y_{\Dt_i}$, is 
\[
\baray{rcl}
L_{p}^{(k)}[y_{\Dt_i}] &\coloneqq& \mathscr{L}_{y_{\Dt_i}}\left( p(\xdt{i})
[\xdt{i}]_{k_1}[\xdt{i}]^T_{k_1} \right).
\earay
\]
In the above, the Riesz functional is applied entry-wise and
\[
k_1 \coloneqq \lfloor k-\deg(p)/2 \rfloor.
\]
In particular, when $p=1$ is the constant one polynomial, we get the {\it moment matrix}
\be\label{momentmx}
M_{\Dt_i}^{(k)}[\ydt{i}] \coloneqq L_{1}^{(k)}[y_{\Dt_i}].
\ee
For a matrix polynomial $G_i\in \mc{S}\re[\xdt{i}]^{\ell \times \ell}$ with entries as
\[
G_i \, \coloneqq  \,
\big(  (G_i)_{st} \big)_{1 \le s,t \le \ell},
\]
its localizing matrix generated by $\ydt{i}$ is the $\ell \times \ell$ block matrix

\be\label{local}
L_{G_i}^{(k)}[y] \, \coloneqq  \,
\big( L_{(G_i)_{st}}^{({k})}[\ydt{i}] \big)_{1 \le s,t \le \ell},
\ee
where 
$
L_{(G_i)_{st}}^{({k})}[\ydt{i}] \, \coloneqq  \,
\mathscr{L}_{\ydt{i}}((G_i)_{st}(\xdt{i}) \cdot [\xdt{i}]_{k_2}[\xdt{i}]_{k_2}^T)$ and
$k_2 \coloneqq \lfloor k - \deg(G_i)/2 \rfloor$.
If $\deg(G_i) \le d$ and $y_{\Dt_i} \in \re^{\N_{d}^{\Dt_i}}$, we define
\be\label{Giyi}
G_i[\ydt{i}]\coloneqq \big(\langle (G_i)_{{st}}, y_{\Dt_i} \rangle
\big)_{1 \le s,t \le \ell}.
\ee
Note that $G_i[\ydt{i}]$ is a principal submatrix of $L_{G_i}^{(k)}[\ydt{i}]$,
consisting of the $(1,1)$-entries of its blocks.
Therefore, if $L_{G_i}^{(k)}[\ydt{i}] \succeq 0$, then $G_i[\ydt{i}] \succeq 0$.

For a given degree $k$, denote the monomial power set
\be\label{u}
\U_k \coloneqq \bigcup_{i =1}^m  \mathbb{N}_{2k}^{\Dt_i}.
\ee
Let $\re^{\U_k}$ denote the space of real vectors labeled such that
\[
y  =  (y_{\af})_{ \af \in \U_k }.
\]
For a given $y\in\re^{\U_k}$, we denote the subvector
\be \label{yDti}
y_{\Dt_i} \, \coloneqq \, (y_{\af})_{ \af \in \mathbb{N}_{2k}^{\Dt_i} }.
\ee
For the objective $f$ as in \reff{spar:matPOP} and $y\in \re^{\U_k}$, we have
\be \label{<f,y>}
\langle f, y \rangle  \, \coloneqq \,
\langle f_1, y_{\Dt_1} \rangle + \cdots + \langle f_m, y_{\Dt_m} \rangle .
\ee

\subsection{Convex matrix polynomials}
\label{sc:pre:cvx}
A matrix polynomial $P(x) \in \mc{S}\re[x]^{\ell \times \ell}$ is said to be {\it convex}
over a convex domain $\mc{D} \subseteq \re^n$
if for all $u, v \in \mc{D}$
and for all $0 \le \lmd \le 1$,
it holds that ($X \preceq Y$ means $Y - X \succeq 0$)
\[
P(\lmd u + (1-\lmd)v) \preceq \lmd P(u) + (1-\lmd)P(v).
\]
If $-P(x)$ is convex over $\mc{D}$, then $P$ is called {\it concave} over $\mc{D}$.
The matrix polynomial $P(x)$ is convex if and only if for all
$\xi\in\re^{\ell}$ and for all $u\in \mc{D}$,
the Hessian matrix $\nabla^2 (\xi^T P(x)\xi)$ is positive semidefinite at $x=u$.
Furthermore, $P(x)$ is said to be {\it SOS-convex}
if for every $\xi \in \re^{\ell}$, there exists a matrix polynomial $Q(x)$ such that
\[
\nabla^2 \left(\xi^T P(x) \xi\right) \, =  \, Q(x)^T Q(x) .
\]
The coefficients of the above $Q(x)$ may depend on $\xi$.
Similarly, if $-P(x)$ is SOS-convex, then $P(x)$ is called {\it SOS-concave}.
We refer to \cite{nie2011PMI} and \cite[Chapter~10.5]{nie2023moment}
for more details about convex matrix polynomials.

\section{Sufficient and necessary conditions for tightness}
\label{sc:char}

In this section, we give a sufficient and necessary condition for
the sparse matrix Moment-SOS hierarchy to be tight for solving \reff{spar:matPOP}.
Denote the degree
\begin{equation}\label{eq:k0}
k_0:=\max_{i\in [m]} \Big( \lceil\deg(f)/2\rceil,  \lceil\deg(G_i)/2\rceil \Big).
\end{equation}
For $k\geq k_0$, the $k$th order sparse matrix SOS relaxation for \reff{spar:matPOP} is
\be  \label{ksos:MPOP}
\left\{ \baray{ccl}
f_k^{spa} \coloneqq  & \max & \gm \\
& \st & f - \gm \in   \qmod{G}_{spa, 2k}.
\earay \right.
\ee
Its dual optimization problem is the $k$th order sparse matrix moment relaxation
\be  \label{kmom:MPOP}
\left\{ \baray{ccll}
f_k^{smo} \coloneqq  & \min & \langle f, y \rangle \coloneqq \langle f_1, \ydt{1} \rangle + \cdots + \langle f_m, \ydt{m} \rangle  \\
& \st &  L_{G_i}^{(k)}[\ydt{i}] \succeq 0 ,\, i = 1, \ldots, m,   \\
&    &  M_{\Dt_i}^{(k)}[\ydt{i}]\succeq 0,\, i = 1, \ldots, m,   \\
&    &  y_0 = 1, \,  y \in \re^{ \U_{k} } .
\earay \right.
\ee
We refer to Section~\ref{sc:pre} for the above notation.
Recall that $f_{\min}$ denotes the minimum value of \reff{spar:matPOP}.
When the running intersection property (RIP) holds, if each $\qm_{\Dt_i}{[G_i]}$ is archimedean,
it is shown in \cite{Kojima09} that $f_k^{spa} \rightarrow f_{\min}$ as $k \rightarrow \infty$.
When $f_k^{spa} = f_{\min}$ for some $k$,
the hierarchy of sparse SOS relaxation \reff{ksos:MPOP}
is said to be {\it tight}. Similarly, if $f_k^{smo} = f_{\min}$ for some $k$,
then the hierarchy of \reff{kmom:MPOP} is {\it tight}.
If they are both {\it tight}, the sparse matrix Moment-SOS hierarchy of
\reff{ksos:MPOP}-\reff{kmom:MPOP} is said to be {\it tight},
or to have {\it finite convergence}.

In the following, we prove a sufficient and necessary condition for
the tightness of the sparse matrix Moment-SOS hierarchy of
\reff{ksos:MPOP}-\reff{kmom:MPOP}.

\begin{theorem}
\label{thm:sepa:match}
Consider the sparse matrix Moment-SOS hierarchy of~\reff{ksos:MPOP}-\reff{kmom:MPOP}.
\begin{enumerate}[(i)]
	
	\item For a relaxation order $k \ge k_0$, it holds
	\be
	\label{eq:f-fmin_in_IQ_2k}
	f-f_{\min} \in \qmod{G}_{spa, 2k}
	\ee
	if and only if there exist sparse polynomials
	$p_i \in \re[\xdt{i}]_{2k}$ such that
	\be
	\label{eq:p+fmin_in_idl}
	\boxed{
		\begin{gathered}
			p_1  + \cdots + p_m  + f_{\min}  =0, \\
			f_i + p_i \in \qm_{\Dt_i}[G_i]_{2k}, \ i=1,\ldots, m.
		\end{gathered}
	}
	\ee
	The equation in the above is equivalent to
	\[
	f - f_{\min}  \, = \, (f_1 + p_1)  + \cdots + (f_m + p_m ).
	\]

	\item When \reff{eq:p+fmin_in_idl} holds for some order $k$,
	the minimum value $f_{\min}$ of (\ref{spar:matPOP})
	is achievable if and only if all sparse polynomials $f_i + p_i$
	have a common zero in $K$, i.e., there exists $u\in K$ such that
	$f_i(\udt{i})+ p_i(\udt{i})=0$ for all $i\in[m]$.
	
\end{enumerate}
\end{theorem}

\begin{proof}
(i) Let $\gm = f_{\min}$, $S = \qm[G]_{spa,2k}$, and
$S_i = \qm_{\Dt_i}[G_i]_{2k}$ for each $i=1,\ldots, m$.
Note that $f = f_{1}+\dots+f_m$ and each $f_i\in\re[\xdt{i}]$.
Observe that
\[
S = S_1 + \cdots + S_m, \quad \text{each} \ S_i \subseteq \re[\xdt{i}].
\]
By \cite[Lemma 2.1]{NQZT24},
it holds that $f-\gamma \in S$
if and only if there exist polynomials $p_i \in \re[\xdt{i}]$ such that
\[
p_1  + \cdots + p_m  + \gamma  =0 , \quad f_i + p_i \in S_i \text{ for each } i.
\]

\smallskip
\noindent (ii) Assume \reff{eq:p+fmin_in_idl} holds for some order $k$. \\
($\Rightarrow$): Suppose $f_{\min}$ is achievable for \reff{spar:matPOP},
then there exists a minimizer $u \in K$ such that $f_{\min} = f(u)$. By the assumption that (\ref{eq:p+fmin_in_idl}) holds, we have
\[
f-f_{\min} =
- \Big( f_{\min}+\sum_{i=1}^m p_i \Big) + \sum_{i=1}^m (f_i + p_i)
\in  \qm[G]_{spa,2k}.
\]
Since $p_1 + \dots + p_m +f_{\min} = 0,$  it holds
\[
\sum_{i=1}^m (f_i(\udt{i}) + p_i(\udt{i})) = 0.
\]
Since each $f_i(\udt{i}) + p_i(\udt{i}) \ge 0$ on $K_{\Dt_i}$,
we have
$
f_i(\udt{i}) + p_i(\udt{i}) = 0
$
for all $i \in [m]$. Therefore,
$u \in K$ is a common zero of all $f_i+p_i$.

\smallskip
\noindent
($\Leftarrow$): Suppose $u\in K$ is a common zero of all $f_i+p_i$, then
\[
\sum\limits_{i=1}^{m} f_i(\udt{i})+ p_i(\udt{i}) = f(u) + p_1(\udt{1}) + \cdots + p_m(\udt{m}) = 0.
\]
Since $p_1  + \cdots + p_m  + f_{\min} =0$,
$f(u) = f_{\min}$, so $f_{\min}$ is achievable.
\end{proof}

We remark that when (\ref{eq:f-fmin_in_IQ_2k}) holds, the polynomials $p_1 \ddd p_m$ satisfying \reff{eq:p+fmin_in_idl} can be selected as 
\[
p_i = \sigma_i - f_i, \quad i = 1 \ddd m,
\] 
where polynomials $\sigma_i\in \qm_{\Dt_i}[G_i]_{2k}$ are such that $f - f_{\min} = \sigma_1+ \cdots +\sigma_m$.
Moreover, such $p_i$ can be explicitly given without using $\sig_i$ when (\ref{spar:matPOP}) satisfies some convexity assumptions; see Theorem~\ref{tm:convex}.
The following is an exposition of the above theorem.

\begin{exm}
Let $\Dt_1 = \{1,2\}$ and $\Dt_2 = \{2,3\}$.
Consider the following sparse matrix polynomial optimization problem
($f_1 = x_1, f_2 = -x_3$)
\be\label{eg_ach}
\left\{
\begin{array}{cl}
	\min\limits_{x \in \re^3} & x_1-x_3 \\
	\st & \bbm 0 & x_1- x_2 \\ x_1 - x_2 & x_2^2-x_1^2 \ebm \succeq 0, \quad
	\bbm 0 & x_2- x_3 \\ x_2 - x_3 & x_3^2-x_2^2 \ebm \succeq 0.
\end{array}\right.
\ee
Clearly, the minimum value $f_{\min} = 0$.
Since there exist polynomials
$p_1 = -x_2$ and $p_2 = x_2$ such that $p_1 + p_2 + f_{\min} = 0$,
\[
\begin{gathered}
	x_1 + p_1 = \frac{1}{8} \bbm {x_1+x_2+2} \\ 2 \ebm^T
	\bbm 0 & x_1- x_2 \\ x_1 - x_2 & x_2^2-x_1^2 \ebm
	\bbm {x_1+x_2+2} \\ 2 \ebm, \\
	- x_3 + p_2 = \frac{1}{8} \bbm {x_2+x_3+2} \\ 2 \ebm^T
	\bbm 0 & x_2- x_3 \\ x_2 - x_3 & x_3^2-x_2^2 \ebm
	\bbm {x_2+x_3+2} \\ 2 \ebm,
\end{gathered}
\]
the sparse SOS relaxation is tight for all $k\ge 2$.
\end{exm}

Theorem~\ref{thm:sepa:match} gives a sufficient and necessary condition for the membership $f-f_{\min} \in \qmod{G}_{spa, 2k}$.
When $f-f_{\min} \in \qmod{G}_{spa, 2k}$ holds for some $k\ge k_0$, the sparse matrix Moment-SOS hierarchy \reff{ksos:MPOP}-\reff{kmom:MPOP} is tight, and the sparse SOS relaxation \reff{ksos:MPOP} achieves its optimal value $f_k^{spa}$.
However, it is possible that the optimal value of \reff{ksos:MPOP} is not achievable, while we still have $f_k^{spa} = f_{\min}$ for some $k\ge k_0$.
This can be shown in the following example.

\begin{exm}\label{nach}
Let $\Dt_1 = \{1,2\}$ and $\Dt_2 = \{2,3\}$.
Consider the following sparse matrix polynomial optimization problem
($f_1=x_1, f_2=-x_3$)
\be\label{eg_nach}
\left\{
\begin{array}{cl}
	\min\limits_{x \in \re^3} &x_1 -x_3 \\
	\st & G_i(\xdt{i}) \succeq 0 , \ i = 1, 2.
\end{array}\right.
\ee
In the above, each $G_i$ is given as follows:
\[
G_i(\xdt{i}) = \bbm 0 & x_i^2+ x_{i+1}^2 \\ x_i^2+ x_{i+1}^2 & x_i^2+ x_{i+1}^2 \ebm.
\]
Clearly, the minimum value $f_{\min} = 0$ and the minimizer $x^* = 0$.
We claim that $f_k^{spa} = 0$ for all $k\ge 1$, since for all $\eps > 0$,
\begin{align*}
	f + \eps &= \frac{\eps}{4} \left[
	\Big(1 + \frac{2x_1}{\eps}\Big)^2+ \Big(1 - \frac{2x_2}{\eps}\Big)^2 + \Big(1 + \frac{2x_2}{\eps}\Big)^2 + \Big(1 - \frac{2x_3}{\eps}\Big)^2 \right] \\
	&+ \frac{1}{\eps} \bbm -1 \\ 1 \ebm^T
	\Big( \bbm 0 & x_1^2+ x_{2}^2 \\ x_1^2+ x_{2}^2 & x_1^2+ x_{2}^2 \ebm
	+ \bbm 0 & x_2^2+ x_{3}^2 \\ x_2^2+ x_{3}^2 & x_2^2+ x_{3}^2 \ebm \Big)
	\bbm -1 \\ 1 \ebm.
\end{align*}
This means that the sparse SOS relaxation is tight.
However, its optimal value is not achievable.
Suppose otherwise that $\gm = 0$ is feasible for the SOS relaxation.
Then,
\[
x_1 - x_3 - 0 = \sum_{i=1}^{2}
\Big[ \sig_i (\xdt{i})+ \sum_{j=1}^{s_i} P_{ij}(\xdt{i})^T G_i(\xdt{i}) P_{ij}(\xdt{i}) \Big],
\]
for some $\sig_i \in \Sig[\xdt{i}]$ and $P_{ij} \in \re[\xdt{i}]^2$.
Let $x_1=x_2=t$ and $x_3=-t$, then we get
\[
2t = \hat{\sig}(t) +  v(t)^T \bbm 0 & 2t^2 \\ 2t^2 & 2t^2 \ebm v(t), \quad \hat{\sig} \in \Sig[t], v \in \re[t]^2.
\]
Plugging in $t=0$, the above implies $\hat{\sig}(0) = 0$. So, $\hat{\sig} = t^2 \cdot \sig_1$ for another SOS polynomial $\sig_1$.
Then, we have
\[
2t = t^2 \cdot \sig_1 + t^2 \cdot v(t)^T \bbm 0 & 2 \\ 2 & 2 \ebm v(t).
\]
However, this is a contradiction because $0$ is a simple root of the left hand side but it is a multiple root of the right hand side.
Therefore, the condition (\ref{eq:p+fmin_in_idl}) does not hold and the optimal value of the SOS relaxation is not achievable.
\end{exm}

The following theorem characterizes the tightness $f_k^{spa} = f_{\min}$ when the optimal value of the sparse SOS relaxation \reff{ksos:MPOP} is not achievable.
\begin{theorem}  \label{thm:sepa:notattain}
The $k$th order sparse SOS relaxation \reff{ksos:MPOP} is tight (i.e., $f_{\min} = f^{spa}_k$)
if and only if for every $\epsilon>0$,
there exist sparse polynomials $p_i \in \re[\xdt{i}]_{2k}$ such that
\be  \label{eq:sep_pi_eps}
\boxed{
	\begin{gathered}
		p_1  + \cdots + p_m  + f_{\min}  =0, \\
		f_i + p_i + \eps \in \qm_{\Dt_i}[G_i]_{2k}, \ i=1,\ldots, m.
	\end{gathered}
}
\ee
\end{theorem}
\begin{proof}
($\Leftarrow$): Suppose that for all $\eps >0$, \reff{eq:sep_pi_eps} holds for some polynomials
$p_i \in  \re[\xdt{i}]_{2k}$. Then
\[
f - (f_{\min} - m\eps) =
-\Big[ f_{\min} + \sum_{i=1}^{m}p_i \Big] + \sum_{i=1}^{m} (f_i+p_i+\eps)
\in  \qm[G]_{spa, 2k}.
\]
This means that $\gm = f_{\min} - m\eps$ is feasible for the $k$th order sparse SOS relaxation
\reff{ksos:MPOP}, so $f_{k}^{spa} \ge f_{\min}-m\eps$.
Since $\eps > 0$ is arbitrary, we get $f_{k}^{spa} \ge f_{\min}$.
On the other hand, we always have $f_{k}^{spa} \le f_{\min}$,
so $f_{k}^{spa} = f_{\min}$.

\medskip \noindent
($\Rightarrow$):
Suppose the relaxation \reff{ksos:MPOP} is tight.
Then, for arbitrary $\eps > 0$, it holds
\[f - (f_{\min} - m\eps) \in \qmod{G}_{spa, 2k}.\]
For each $i$, let
$
f^{\epsilon}_i(\xdt{i}) \, \coloneqq \,  f_i(\xdt{i})+\eps,
$
then
\[
f - (f_{\min} - m\epsilon) = \Big(\sum_{i=1}^mf^{\epsilon}_i \Big) -f_{\min} \in \qmod{G}_{spa,2k}.
\]
Applying \cite[Lemma 2.1]{NQZT24},
we let $\gm \coloneqq f_{\min}$, $f_i = f_i^{\eps}$, $S = \qm[G]_{spa,2k}$, and
$S_i = \qm_{\Dt_i}[G_i]_{2k}$ for each $i=1,\ldots, m$.
So there exist polynomials $p_i  \in \re[\xdt{i}]_{2k}$
such that $p_1 + \cdots + p_m + f_{\min} =0$ and for all $i$,
\[
f_i^{\eps} + p_i = f_i + p_i + \eps \in \qm_{\Dt_{i}}[G_i]_{2k}  .
\]
So, \reff{eq:sep_pi_eps} holds.
\end{proof}

We remark that for problem~(\ref{eg_nach}) in Example~\ref{nach}, the sparse matrix SOS relaxation \reff{ksos:MPOP} is tight for all $k\ge 1$.
This is because for $p_1 = -x_2$ and $p_2 = x_2$, we have $p_1 + p_2 + f_{\min} = 0$, and
\[
f_1 + p_1 + \eps = \frac{\eps}{2} \left[
\Big(1 + \frac{x_1}{\eps}\Big)^2+ \Big(1 - \frac{x_2}{\eps}\Big)^2 \right] +
\frac{1}{2\eps} \bbm -1 \\ 1 \ebm^T
\bbm 0 & x_1^2+ x_{2}^2 \\ x_1^2+ x_{2}^2 & x_1^2+ x_{2}^2 \ebm
\bbm -1 \\ 1 \ebm,
\]
\[
f_2 + p_2 + \eps =  \frac{\eps}{2} \left[
\Big(1 + \frac{x_2}{\eps}\Big)^2+ \Big(1 - \frac{x_3}{\eps}\Big)^2 \right] +
\frac{1}{2\eps} \bbm -1 \\ 1 \ebm^T
\bbm 0 & x_2^2+ x_{3}^2 \\ x_2^2+ x_{3}^2 & x_2^2+ x_{3}^2 \ebm
\bbm -1 \\ 1 \ebm.
\]

\section{Detecting tightness and extracting minimizers}
\label{sc:flat}

Theorems~\ref{thm:sepa:match} and \ref{thm:sepa:notattain} characterize tightness of the sparse matrix Moment-SOS  hierarchy of (\ref{ksos:MPOP})-(\ref{kmom:MPOP}).
In this section, we discuss how to detect the tightness $f^{smo}_k = f_{\min}$
and get minimizers of (\ref{spar:matPOP}).

Let $y^*$ be a minimizer of the sparse moment relaxation (\ref{kmom:MPOP}).
We say that $y^*$ satisfies the {\it flat truncation} condition (see \cite{nie2023moment}) if there exists an integer $t\in [k_0,k]$ such that for all $i=1\ddd m$,
it holds (let $d_i\coloneqq\lceil \deg(G_i)/2\rceil$)
\be\label{eq:ft}
r_i\coloneqq \rank\, M^{(t)}_{\Dt_i}[\ydt{i}^*] = \rank\, M^{(t-d_i)}_{\Dt_i}[\ydt{i}^*].
\ee
We refer to \reff{momentmx} and \reff{yDti} for the notation $\ydt{i}^*$ and the moment matrix $M^{(t)}_{\Dt_i}[\ydt{i}^*]$.
When (\ref{eq:ft}) holds, there exist support sets
\be\label{chi}
\mc{X}_{\Dt_i}\coloneqq \{u^{(i,1)}\ddd u^{(i,r_i)}  \}\subseteq K_{\Dt_i}
\ee
and scalars $\lmd_{i,1}\ddd \lmd_{i,r_i}$ such that for each $i$, it holds
\be \label{decomp_i}
\boxed{
\begin{gathered}
	{y^*_{\Dt_{i}} |_{2t} = \lambda_{i,1}[u^{(i,1)}]_{2t} + \cdots + \lambda_{i,r_i}[u^{(i,r_i)}]_{2t},} \\
	\lambda_{i,1} > 0, \ldots, \lambda_{i,r_i} > 0, \quad
	\lambda_{i,1} + \cdots + \lambda_{i,r_i} = 1.
\end{gathered}
}
\ee
In the above, $y_{\Dt_i}|_{2t}$ denotes the degree-$2t$ truncation of $\ydt{i}$:
\[
y_{\Dt_i}|_{2t} \, \coloneqq \, (y_{\af})_{ \af \in \mathbb{N}_{2t}^{\Dt_i} }.
\]
This is shown in \cite[Theorem~2.4]{henrion2006convergent}.
A numerical method for computing points $u^{(i,1)}\ddd u^{(i,r_i)}$ is introduced in \cite{henrion2005detecting} (also see \cite{nie2023moment}).
Furthermore, if $x^*$ satisfies $\xdt{i}^*\in\mc{X}_{\Dt_i}$ for every $i$, then $f(x^*) = f_{k}^{smo}$ implies that $f_{k}^{smo} = f_{\min}$ and $x^*$ is a minimizer of (\ref{spar:matPOP}).
This is because if $\xdt{i}^*\in\mc{X}_{\Dt_i}$ for each $i$, then $x^*$ must be feasible for (\ref{spar:matPOP}), thus $f(x^*)\ge f_{\min}$, while $f_{k}^{smo}$ is an lower bound for $f_{\min}$.

In the following, we show that if (\ref{eq:ft}) holds with $f^{smo}_t = f^{smo}_k $, then $f(x^*) = f_{k}^{smo}= f_{\min}$ holds for all $x^*$ such that each $\xdt{i}^*\in\mc{X}_{\Dt_i}$.

\begin{theorem}\label{tm:getminimizer}
Let $y^*$ be a minimizer of (\ref{kmom:MPOP}).
Suppose there exists an integer $t\in [k_0,k]$ such that the flat truncation condition (\ref{eq:ft}) holds and $f^{smo}_t = f^{smo}_k$.
\begin{enumerate}
	\item[(i)] If $x^*$ satisfies $\xdt{i}^*\in \mc{X}_{\Dt_i}$ for all $i=1\ddd m$,
	then $f_{\min} = f^{smo}_k$ and $x^*$ is a minimizer of (\ref{spar:matPOP}).
	
	\item[(ii)] In item (i), if, in addition, there is no duality gap between \reff{ksos:MPOP} and \reff{kmom:MPOP}, then $f_{\min} = f_k^{spa}$.
\end{enumerate}
\end{theorem}
\begin{proof}
Since $y^*$ satisfies the flat truncation condition (\ref{eq:ft}), the decomposition (\ref{decomp_i}) holds for all $i=1\ddd m$.
So, there exist positive scalars
$\rho_1\ddd \rho_m$ and moment matrices $W_{\Dt_i}\succeq 0$ such that for each $i=1\ddd m$,
\[
M_{\Dt_i}^{(t)}[\ydt{i}^*] = \rho_i[\xdt{i}^*]_t[\xdt{i}^*]_t^T + W_{\Dt_i}.
\]
Let $\rho\coloneqq \min\{\rho_1\ddd \rho_m\}$, and let $\hat{y}\in\re^{\mathbb{U}_k}$ be the tms such that $\hat{y}_{\alpha} = (x^*)^{\alpha}$ for all $\alpha \in \mathbb{U}_k$.
The subvector $\hat{y}|_{2t}$ is feasible for (\ref{kmom:MPOP})
with the relaxation order equal to $t$, since every $\xdt{i}^*\in\mc{X}_{\Dt_i}\subseteq K_{\Dt_i}$.

For the case that $\rho = 1$, it is clear that every $\mc{X}_{\Dt_i} = \{\xdt{i}^*\}$,
and the conclusion follows directly.
In the following, we consider the case that $\rho<1$. Let
\[
\tilde{y}\coloneqq (y^*-\rho\hat{y})/(1-\rho).
\]
Then, for each $i$, it holds
\[ L_{G_i}^{(t)}[\tilde{y}] = \frac{1}{1-\rho}(L_{G_i}^{(t)}[y^*] - \rho L_{G_i}^{(t)}[\hat{y}]) \succeq 0,\]
by (\ref{decomp_i}) and the fact that $\rho \le \rho_i\le \lmd_{i,\hat{j}}$, where $\hat{j}$
is the label such that $\xdt{i}^* = u^{(i,\hat{j})}$.
Similarly, one can show that $M_{\Dt_i}^{(t)}[\tilde{y}]\succeq 0$.
So, $\tilde{y}|_{2t}$ is also feasible for (\ref{kmom:MPOP}) with the relaxation order equal to $t$.
Therefore, by the assumption that $f^{smo}_t = f^{smo}_k$, we have
\[
f^{smo}_t = \lip f,y^* \rip = \lip f,\hat{y}\rip = f(x^*).
\]
This completes the proof.
\end{proof}

By Theorem~\ref{tm:getminimizer}, once we get a minimizer $y^*$ for (\ref{kmom:MPOP}),
we may check whether the moment relaxation is tight and extract minimizers by checking (\ref{eq:ft}).
Moreover, since we usually solve (\ref{kmom:MPOP}) with an increasing relaxation order $k$,
we can use the optimal value of moment relaxations
with lower relaxation orders to check if $f^{smo}_t = f^{smo}_k$ holds or not.

Summarizing the above, we get the following algorithm
for solving the sparse matrix polynomial optimization problem (\ref{spar:matPOP}).

\begin{alg} \label{alg:getminimizer}
For (\ref{spar:matPOP}),
let $k_0$ be as in (\ref{eq:k0}) and $k\coloneqq k_0$.
Let
$
d \coloneqq \max \{d_1 \ddd d_m\}.
$
Do the following:
\begin{description}

\item [Step~1:] Solve the sparse moment relaxation (\ref{kmom:MPOP}) of order $k$
for a minimizer $y^*$. Let $t \coloneqq d$.

\bigskip
\item [Step~2:] For each $i=1\ddd m$, check whether the flat truncation condition \reff{eq:ft} holds or not.
If it holds for all $i$, extract the points $u^{(i,1)}\ddd u^{(i,r_i)}$ satisfying (\ref{decomp_i})
and go to Step~4.

\bigskip
\item [Step~3:] If \reff{eq:ft} does not hold for some $i$, update $t\coloneqq t+1$.
If $t \le k$, go to Step~2;
if $t>k$, let $k \coloneqq k+1$ and go to Step~1.

\bigskip
\item [Step~4:] Let $\mc{X}_{\Dt_i}$ be as in \reff{chi} and formulate the set
\be
\mc{X}\coloneqq \{ x\in\re^n: \xdt{i}\in \mc{X}_{\Dt_i},\, i=1\ddd m \}.
\ee
If $\mc{X} \ne \emptyset$ and $f^{smo}_t = f^{smo}_k$, output that $\mc{X}$
is the set of minimizers for (\ref{spar:matPOP}) and stop;
otherwise, let $k \coloneqq k+1$ and go to Step~1.

\end{description}

\end{alg}

\begin{remark}\label{rmk:ftrank}
In Step~2 of Algorithm~\ref{alg:getminimizer},
we need to evaluate the ranks of $M^{(t)}_{\Dt_i}[\ydt{i}^*]$ and $M^{(t-d_i)}_{\Dt_i}[\ydt{i}^*]$ to check the flat truncation condition (\ref{eq:ft}).
Evaluating matrix ranks is a classical problem in numerical linear algebra.
It may be difficult to determine the rank accurately when a matrix is near to be singular, due to round-off errors.
In computational practice, the rank of moment matrices is often determined as follows.
Suppose $\mu_1 \ge \mu_2\ge \cdots \ge \mu_{s}>0$ are the positive eigenvalues of the psd matrix $M^{(t)}_{\Dt_i}[\ydt{i}^*]$.
If there exists a smallest $r<s$ such that $\mu_{r+1} < \varepsilon \mu_r$ for some given tolerance $\varepsilon$ (typically, $\varepsilon \coloneqq 10^{-3}$), then $r$ is evaluated as the rank of $M^{(t)}_{\Dt_i}[\ydt{i}^*]$; otherwise, its rank is evaluated as $s$.
Typically, there exists a significant gap between $\mu_r$ and $\mu_{r+1}$ when flat truncation holds; see Examples~\ref{ex:nonconvex_hol} and \ref{ex:from_henrion}.
The rank of $M^{(t-d_i)}_{\Dt_i}[\ydt{i}^*]$ can be evaluated in a similar way.
Moreover, due to the existence of round-off errors, we also numerically check if the extracted minimizers are feasible or not, and whether the computed lower bound $f^{smo}_k$ is attained at these points, up to an apriori given tolerance (say, $10^{-3}$).
This method of checking the flat truncation condition and extracting minimizers is implemented in the software {\tt GloptiPoly 3} \cite{GloPol3}, {\tt YALMIP} \cite{yalmip}, and {\tt TSSOS} \cite{TSSOS}.
Moreover, there exists work on the robustness of extracting minimizers using flat truncation; see \cite{Klep2018}. 
\end{remark}

\begin{exm}
Let $\Dt_1 = \{1,2\}$ and $\Dt_2 = \{2,3\}$.
Consider the sparse matrix polynomial optimization problem
($f_1 \coloneqq -x_1-4x_2^2,\ f_2 \coloneqq -x_3$)
\be\label{eg_flat}
\left\{
\begin{array}{cl}
	\min\limits_{x \in \re^3} & -x_1-4x_2^2-x_3 \\
	\st & \bbm 1+x_1 & x_2^2  \\
	x_2^2 & 1-x_1  \ebm \succeq 0, \quad
	\bbm 1+x_3 & x_2^2  \\
	x_2^2 & 1-x_3  \ebm \succeq 0.
\end{array}
\right.
\ee
The sparse moment relaxation \reff{kmom:MPOP} can be implemented in \verb|YALMIP|.
For $k=3$, the relaxation \reff{kmom:MPOP} is tight, and we get $f_3^{smo} = f_{\min} = -\frac{10}{\sqrt{5}}$.
By solving \reff{kmom:MPOP}, we get
\[{\small
	M^{(3)}_{\Dt_1} [\ydt{1}^*] =
	\bbm
	1 &  \frac{1}{\sqrt{5}}  &  0  &  \frac{1}{5}  &  0  &  \frac{2}{\sqrt{5}}  &  5^{-\frac{3}{2}}  &  0  & \frac{2}{5}  &  0 \\
	\frac{1}{\sqrt{5}}   &  \frac{1}{5}   &  0  &   5^{-\frac{3}{2}}  &  0  &  \frac{2}{5}  &  \frac{1}{25}  &  0  & \frac{2}{5\sqrt{5}}  &  0 \\
	0  &  0 &   \frac{2}{\sqrt{5}}  &  0  &  \frac{2}{5} &   0 &   0  &  \frac{2}{5\sqrt{5}}  &  0 & \frac{4}{5} \\
	\frac{1}{5}  &   5^{-\frac{3}{2}}  &  0  &   \frac{1}{25}  &  0  &  \frac{2}{5\sqrt{5}}  &  5^{-\frac{5}{2}}  &  0 &   \frac{2}{25}  & 0 \\
	0  &  0  &  \frac{2}{5}  &  0 &   \frac{2}{5\sqrt{5}}  &  0  &  0  &  \frac{2}{25}  &  0 &  \frac{4}{5\sqrt{5}} \\
	\frac{2}{\sqrt{5}}  &  \frac{2}{5}  &  0  &  \frac{2}{5\sqrt{5}}  &  0  &  \frac{4}{5}  &  \frac{2}{25}  &  0  &  \frac{4}{5\sqrt{5}} &  0 \\
	5^{-\frac{3}{2}}  &   \frac{1}{25}  &  0  &  5^{-\frac{5}{2}}  &  0  &  \frac{2}{25}  &  0.6917  &  0  &  \frac{2}{25\sqrt{5}} &  0 \\
	0  &  0  &  \frac{2}{5\sqrt{5}}  &  0  &  \frac{2}{25} &  0  &  0  &  \frac{2}{25\sqrt{5}}  &  0 &  \frac{4}{25} \\
	\frac{2}{5}  &  \frac{2}{5\sqrt{5}} &   0  &  \frac{2}{25} &   0  &  \frac{4}{5\sqrt{5}} &   \frac{2}{25\sqrt{5}} &   0  &  \frac{4}{25}  & 0 \\
	0  &  0  &  \frac{4}{5}  &  0  &  \frac{4}{5\sqrt{5}}  &  0  &  0  &  \frac{4}{25}  &  0 & \frac{8}{5\sqrt{5}} \\
	\ebm,}
\]
\[{\small
	M^{(3)}_{\Dt_2} [\ydt{2}^*] =
	\bbm
	1  &  0  &  \frac{1}{\sqrt{5}}   &  \frac{2}{\sqrt{5}}  &  0  &  \frac{1}{5}  &  0  &  \frac{2}{5}  &  0  &  5^{-\frac{3}{2}} \\
	0  &  \frac{2}{\sqrt{5}}  &  0  &  0  &  \frac{2}{5}  &  0  &   \frac{4}{5}  &  0  &   \frac{2}{5\sqrt{5}}  & 0 \\
	\frac{1}{\sqrt{5}}  &   0  &  \frac{1}{5}  &  \frac{2}{5}  &  0  &  5^{-\frac{3}{2}}  &  0  &   \frac{2}{5\sqrt{5}}  &  0 &  \frac{1}{25} \\
	\frac{2}{\sqrt{5}}  &  0  &  \frac{2}{5}  &   \frac{4}{5}  &  0  &   \frac{2}{5\sqrt{5}}  &  0  &  \frac{4}{5\sqrt{5}} &   0 &   \frac{2}{25} \\
	0  &  \frac{2}{5}  &  0  &  0  &   \frac{2}{5\sqrt{5}}  &  0  &  \frac{4}{5\sqrt{5}}  &  0  &   \frac{2}{25}  & 0 \\
	\frac{1}{5}  &  0  &  5^{-\frac{3}{2}}  &   \frac{2}{5\sqrt{5}}   & 0  &  \frac{1}{25}  &  0  &   \frac{2}{25}  &  0 &  5^{-\frac{5}{2}} \\
	0  &   \frac{4}{5}  &  0  &  0   & \frac{4}{5\sqrt{5}}  &  0  &  \frac{8}{5\sqrt{5}}  &  0  &  \frac{4}{25}  &  0 \\
	\frac{2}{5}  &  0  &   \frac{2}{5\sqrt{5}}  &  \frac{4}{5\sqrt{5}}  &  0  &   \frac{2}{25}  &  0  &  \frac{4}{25}  &  0  &  \frac{2}{25\sqrt{5}} \\
	0  &   \frac{2}{5\sqrt{5}}  &  0  &  0  &   \frac{2}{25}  &  0  &  \frac{4}{25}  &  0  &  \frac{2}{25\sqrt{5}}  &  0 \\
	5^{-\frac{3}{2}}  &  0  &  \frac{1}{25}  &   \frac{2}{25}  &  0  &  5^{-\frac{5}{2}}  &  0  &  \frac{2}{25\sqrt{5}}  &  0  & 0.6917
	\ebm.}
\]
The flat truncation \reff{eq:ft} holds for $t=2$ (but not for $t=3$) since
\[
\begin{gathered}
	\rank \, M^{(3)}_{\Dt_1}[\ydt{1}^*] = \rank \, M^{(3)}_{\Dt_2}[\ydt{2}^*] = 3, \\
	\rank \, M^{(1)}_{\Dt_1}[\ydt{1}^*] = \rank \, M^{(2)}_{\Dt_1}[\ydt{1}^*] = 2, \\
	\rank \, M^{(1)}_{\Dt_2}[\ydt{2}^*] = \rank \, M^{(2)}_{\Dt_2}[\ydt{2}^*] = 2.
\end{gathered}
\]
By Theorem~\ref{tm:getminimizer}, we get two minimizers
$(\frac{1}{\sqrt{5}}, \pm \sqrt{\frac{2}{\sqrt{5}}}, \frac{1}{\sqrt{5}}).$

\end{exm}

\section{Convex sparse matrix constrained optimization}
\label{sc:convex}

This section discusses convex sparse matrix polynomial optimization.
We consider that each $f_i$ is convex in $\xdt{i}$ and each $G_i$ is concave in $\xdt{i}$.

Suppose $u\in K$ is a minimizer of \reff{spar:matPOP}.
Recall that $\nabla G_i(\udt{i})$ denotes the derivative of
$G_i$ at $\udt{i}$, given by (\ref{eq:deri_G}),
and $\nabla G_i(\udt{i})^*$ denotes the adjoint mapping
of $\nabla G_i(\udt{i})$, given by (\ref{eq:adj_deri_G}).
If the nondegeneracy condition (\ref{eq:NDC}) holds at $u$,
or the {\it Slater's condition}\footnote{
Slater's condition is said to hold for (\ref{spar:matPOP})
if there exists $u$ such that $G_i(\udt{i})\succ0$ for all $i=1\ddd m$.}
holds, then there exist Lagrange multiplier matrices $\Lmd_i \in \mc{S}^{\ell_i}$ such that
\be\label{FOOC}\left\{\
\begin{gathered}
\nabla f(u) = \sum\limits_{i=1}^{m} \nabla G_i(\udt{i})^*[\Lmd_i], \\
\Lmd_i \succeq 0, \enspace G_i(\udt{i}) \succeq 0, \ i\in [m], \\
\langle \Lmd_i , G_i(\udt{i}) \rangle = 0, \ i \in [m].
\end{gathered}\right.
\ee
The above is called the {\it first order optimality condition} (FOOC).
We refer to \cite{Shapiro97, Sun06} for more details about optimality conditions of nonlinear semidefinite optimization.
For the above $\Lambda_i$, define the Lagrange function
\be\label{eq:Lagrangian}
\mc{L}_i(x) \, \coloneqq \, f_i(\xdt{i}) - \langle \Lmd_i , G_i(\xdt{i}) \rangle.
\ee
Denote the symmetric $n_i$-by-$n_i$ matrix $H_i$,
with entries
\[
(H_i)_{st} \coloneqq 2\big\langle \Lmd_i, \nabla_{x_s}G_i(\udt{i})G_i(\udt{i})^{\dag}\nabla_{x_t}G_i(\udt{i}) \big\rangle,
\]
for $s,t \in \Dt_i$.
In the above, $\nabla_{x_s}$ denotes the partial derivative with respect to $x_s$, and the superscript $\dag$ denotes the Moore-Penrose inverse.
Define
\[
\mc{N}_i \coloneqq \Big\{ v = (v_j)_{j \in \Dt_i} : \sum_{j \in \Dt_i} v_j \cdot E^T \nabla_{x_j} G_i(\udt{i}) E = 0 \Big\},
\]
where $E$ is a matrix whose columns form a basis of $\ker \, G_i(\udt{i})$.

\begin{theorem}\label{tm:convex}
Suppose $u$ is a minimizer of \reff{spar:matPOP} and the FOOC (\ref{FOOC}) holds.
Assume each $f_{i}$ is convex and each $G_i \in \mc{S}\re[\xdt{i}]^{\ell_i \times \ell_i}$ is concave. Then,
\begin{enumerate}[(i)]
	\item There exist sparse polynomials $p_i\in \re[x_{\Delta_i}]$ such that
	\be  \label{eq:f+p>=0}
	\boxed{
		\begin{array}{c}
			\ p_1  + \cdots + p_m  + f_{\min}  =0 , \\
			f_i + p_i \ge 0~ \text{on} ~ K_{\Dt_i}, \,  i=1\ddd m.
	\end{array}}
	\ee
	
	\item Suppose each $\qm_{\Dt_i}[G_i]$ is archimedean.
	Assume that for each $i$, the NDC (\ref{eq:NDC}) for $G_i$ holds at $\udt{i}$, $\rank\, G_i(\udt{i})+\rank\,\Lambda_i = \ell_i$, and
	\[
	v^T (\nabla^2_{\xdt{i}} \mc{L}_i(u) + H_i )v>0, \quad
	\forall \, 0 \ne v\in \mc{N}_i.
	\]
	Then, the sparse matrix Moment-SOS hierarchy of \reff{ksos:MPOP}-\reff{kmom:MPOP}
	is tight and we have
	\[
	f-f_{\min}\in \qm[G]_{spa, 2k}
	\]
	for all $k$ big enough.
\end{enumerate}
\end{theorem}
\begin{proof}
(i) For each $i=1,\ldots, m$, let
\be\label{eq:pi_convex}
p_i(x) \, \coloneqq \, -(x-u)^T\big(\nabla f_i(\udt{i}) - \nabla G_i(\udt{i})^*[\Lmd_i] \big) - f_i(\udt{i}).
\ee
Since $f_i$ and $G_i$ only depend on $\xdt{i}$, we have $p_i \in \re[\xdt{i}]$.
Note that
\[
f_{\min} = \sum_{i=1}^{m} f_i(\udt{i}), \quad \nabla f(u) =  \sum_{i=1}^{m} \nabla f_i(\udt{i}).
\]
For the above $p_i$, the first equation in \reff{FOOC} implies
\begin{align*}
	&  p_1 + \cdots + p_m + f_{\min} \\
	= &-(x-u)^T \Big[ \nabla f(u) - \sum_{i=1}^{m}  \nabla G_i(\udt{i})^*[\Lmd_i]    \Big] - \sum_{i=1}^{m} f_i(\udt{i}) + f_{\min} \\
	= &-(x-u)^T 0 + 0 = 0.
\end{align*}
Note that for each $i=1\ddd m$, it holds
\[
\begin{aligned}
	\nabla f_i(\udt{i}) + \nabla p_i(\udt{i}) &= \nabla f_i(\udt{i}) -\big( \nabla f_i(\udt{i}) - \nabla G_i(\udt{i})^*[\Lmd_i] \big) \\
	& = \nabla G_i(\udt{i})^*[\Lmd_i].
\end{aligned}
\]
Thus, by the assumption that (\ref{FOOC}) holds, each $\udt{i}$ satisfies the FOOC for
\be\label{eq:minf+p} \left\{
\begin{array}{cl}
	\displaystyle\min_{\xdt{i} \in\re^{\Dt_i}} & f_i(\xdt{i})+p_i(\xdt{i}) \\
	\st & G_i(\xdt{i}) \succeq 0.
\end{array}
\right.
\ee
By (\ref{eq:pi_convex}), each $p_i(\xdt{i})$ is an affine function in $\xdt{i}$.
So, $f_i(\xdt{i}) + p_i(\xdt{i})$ is convex in $\xdt{i}$, and $\udt{i}$ is a minimizer of (\ref{eq:minf+p}).
Moreover, we can see that
\[
f_i(\udt{i}) + p_i(\udt{i}) = 0.
\]
Thus, the minimum value of \reff{eq:minf+p} is $0$ and $f_i+p_i\ge0 $ on $K_{\Dt_i}$.
Therefore, \eqref{eq:f+p>=0} holds.

(ii)
For each $i$, note that $\mc{L}_i(x)$ is the Lagrange function for the optimization problem (\ref{eq:minf+p}).
By the given assumption, the nondegeneracy condition, strict complementarity condition, and second order sufficient condition all hold at $\udt{i}$ for (\ref{eq:minf+p}).
Let $p_1\ddd p_m$ be the polynomials in item (i).
By \cite[Theorem~1.1]{Huang2024}, there exists $k_i\in \N$ such that
\[
f_i+p_i\in \qm_{\Dt_i}[G_i]_{2k_i}.
\]
Since $p_1+\cdots + p_m+f_{\min} = 0$ by item~(i), Theorem~\ref{thm:sepa:match}(i) implies that
\[
f-f_{\min}\in \qm[G]_{spa, 2k},
\]
for all $k\ge\max\{k_1\ddd k_m\}$.
Therefore, the sparse matrix Moment-SOS hierarchy of \reff{ksos:MPOP}-\reff{kmom:MPOP} is tight.
\end{proof}

Now we consider the special case that (\ref{spar:matPOP}) is SOS-convex.
We refer to Section~\ref{sc:pre:cvx} for the SOS-convexity/concavity.
Recall that $k_0$ is given in \reff{eq:k0}.

\begin{theorem}\label{tm:sosconvex}
Suppose $u$ is a minimizer of \reff{spar:matPOP}.
Assume each $f_i$ is SOS-convex and each $G_i$ is SOS-concave. Then,
\begin{itemize}
	\item[(i)]
	For all $k\geq k_0$, we have $f^{smo}_k=f_{\min}$. Moreover, if Slater's condition holds, then $f_k^{spa}=f_{\min}$ and
	\begin{equation}\label{f-fminqmG}
		f-f_{\min}\in \qm[G]_{spa, 2k}.
	\end{equation}
	\item[(ii)] For every minimizer $y^*$ of~\eqref{kmom:MPOP}, the point
	$x^* \coloneqq (y^*_{e_1},\ldots,y^*_{e_n})$
	is a minimizer of~\eqref{spar:matPOP}.
\end{itemize}
\end{theorem}

\begin{proof} (i) Suppose $y$ is a feasible solution of the relaxation \reff{kmom:MPOP} and let $u \coloneqq (y_{e_1}, \ldots, y_{e_n})$.
For each $i$, pick an arbitrary $\xi \in \re^{\ell_i}$. 
Then the scalar polynomial
\[
g_{\xi}(\xdt{i}) \coloneqq \xi^T G_i(\xdt{i}) \xi
\]
is SOS-concave in $\xdt{i}$ because each $G_i$ is SOS-concave by the assumption.
Since each $f_i$ is SOS-convex, by Jensen's inequality (see \cite{Las09} or \cite[Chap.~7]{nie2023moment}), we have
\be\label{jen_ineq}
f_i(\udt{i}) \le \langle f_i , \ydt{i} \rangle, \quad -g_{\xi}(\udt{i}) \le \langle -g_{\xi} , \ydt{i} \rangle.
\ee
The second inequality of \reff{jen_ineq} implies that
\[
\xi^T G_i(\udt{i}) \xi = g_{\xi}(\udt{i}) \ge  \langle g_{\xi} , \ydt{i} \rangle = \xi^T G_i[\ydt{i}] \xi .
\]
Since $y$ is feasible for (\ref{kmom:MPOP}), the localizing matrix
$
L_{G_i}^{(k)}[\ydt{i}] \succeq 0.
$
Note that $G_i[\ydt{i}]$ is a principal sub-matrix of $L_{G_i}^{(k)}[\ydt{i}]$, so
\[
G_i[\ydt{i}] \succeq 0 \quad \text{and hence} \quad \xi^T G_i[\ydt{i}] \xi \ge 0.
\]
Since $\xi$ is arbitrary, we know $G_i(\udt{i}) \succeq 0$.
This is true for all $i$, hence $u$ is a feasible point for \reff{spar:matPOP}.
Also, by the first inequality of \reff{jen_ineq},
\be\label{eq:sum_jensen}
f(u) = \sum_{i=1}^{m} f_i(\udt{i}) \le \sum_{i=1}^{m} \langle f_i , \ydt{i} \rangle = \langle f , y \rangle.
\ee
The above holds for all $y$ that is feasible for \reff{kmom:MPOP},
so $f_{\min} \le f_k^{smo}$.
On the other hand, we always have $f_k^{smo} \le f_{\min}$.
Therefore, $f_k^{smo} = f_{\min}$.

Furthermore, when Slater's condition holds, the moment relaxation~\eqref{kmom:MPOP} has strictly feasible points (see Theorem~2.5.2 of \cite{nie2023moment}). So, the strong duality holds between~\eqref{ksos:MPOP} and~\eqref{kmom:MPOP}, and~\eqref{ksos:MPOP} achieves its optimal value. Therefore, $f_k^{spa} = f_k^{smo} = f_{\min}$ and~\eqref{f-fminqmG} holds.

(ii) Let $y^*$ be a minimizer of~\eqref{kmom:MPOP}.
Then $\lip f, y^*\rip \ge f(x^*)$ by (\ref{eq:sum_jensen}) and $x^*$ is feasible for \reff{spar:matPOP} . Therefore, we have
\[
f_{\min} = f_k^{smo} = \lip f, y^*\rip \ge f(x^*) \ge f_{\min},
\]
which forces $f(x^*) = f_{\min}$. So $x^*$ is a minimizer of~\eqref{spar:matPOP}.
\end{proof}

\begin{exm}\label{soscon}
Let $\Dt_1 = \{1,2,3\}$ and $\Dt_2 = \{2,3,4\}$. Consider the sparse matrix polynomial optimization
\be\label{eg_soscon}
\left\{
\begin{array}{cl}
	\min\limits_{x \in \re^4} & f_1(\xdt{1}) + f_2(\xdt{2}) \\
	\st & G_i(\xdt{i}) \succeq 0, \ i = 1,2.
\end{array}
\right.
\ee
In the above, each
\[
f_i(\xdt{i}) = x_i^4 + 2x_{i+1}^4 + x_{i+2}^4 + 2x_{i+1}^2(x_i^2 + x_{i+2}^2) + x_i + x_{i+1} + x_{i+2},
\]
\[
G_i(\xdt{i}) =
\left[
\begin{array}{rrrrr}
	1-x_i^2 - x_{i+2}^2 && x_ix_{i+1} &&  x_ix_{i+2} \\
	x_ix_{i+1} & & 1-x_{i+1}^2-x_i^2 & & x_{i+1}x_{i+2} \\
	x_ix_{i+2} &&  x_{i+1}x_{i+2} & &  1-x_{i+2}^2-x_{i+1}^2
\end{array}
\right].
\]
Observe that each $f_i$ can be written as
\[
f_i(\xdt{i}) = \bbm x_i^2 \\ x_{i+1}^2 \\ x_{i+2}^2 \ebm^T
\bbm 1 & 1 & 0 \\
1 & 2 & 1 \\
0 & 1 & 1
\ebm
\bbm x_i^2 \\ x_{i+1}^2 \\ x_{i+2}^2 \ebm + \text{linear terms}.
\]
As shown in Example~7.1.4 in \cite{nie2023moment}, we know that $f_i$ is SOS-convex,
since the matrix in the middle is psd and has nonnegative entries.
The matrix $G_i$ is SOS-concave since $\nabla^2(-\xi^T G_i \xi) \succeq 0$
for all $\xi \in \re^3$. This is because the bi-quadratic form
\begin{align*}
	\frac{1}{2}z^T \nabla^2(-\xi^T G_i \xi) z &= z_1^2\xi_1^2 +  z_2^2\xi_2^2 +  z_3^2\xi_3^2
	-2(z_1z_2\xi_1\xi_2 + z_2z_3\xi_2\xi_3 + z_3z_1\xi_3\xi_1 ) \\
	& \quad + z_1^2\xi_2^2 +  z_2^2\xi_3^2 +  z_3^2\xi_1^2
\end{align*}
is nonnegative everywhere (see Section~4 of \cite{ChoLam77}).
The sparse matrix Moment-SOS hierarchy of \reff{ksos:MPOP}-\reff{kmom:MPOP} is tight for $k=2$.
We get $f_{\min} = f_2^{smo} \approx -2.0731$ and the minimizer
$(-0.5361,-0.4230,-0.4230,-0.5361)$.
\end{exm}

\section{Numerical Experiments}
\label{sc:ne}
This section provides numerical experiments for
the sparse matrix Moment-SOS hierarchy of (\ref{ksos:MPOP})-(\ref{kmom:MPOP}).
For all examples in this section, we use {\tt YALMIP} \cite{yalmip} to implement sparse matrix Moment-SOS relaxations.
Moreover, we apply {\tt Gloptipoly 3} \cite{GloPol3} to check flat truncation conditions and extract minimizers.
All semidefinite programs are solved by the software {\tt Mosek} \cite{mosek}.
For the flat truncation condition (\ref{eq:ft}), we evaluate ranks of moment matrices using the method described in Remark~\ref{rmk:ftrank} with $\varepsilon = 10^{-3}$.
The computation is implemented in MATLAB 2023b,
in an Apple MacBook Pro Laptop in MacOS 14.2.1 with 12$\times$Apple M3 Pro CPU and RAM 18GB.
For neatness, only four decimal digits are displayed for computational results.

For Example~\ref{ex:exp_sol}, we solve the sparse matrix Moment-SOS relaxation (\ref{ksos:MPOP})-(\ref{ksos:MPOP}). When $k=3$, the flat truncation condition (\ref{eq:ft}) holds for $t=1$ and $r_i=1$ for each $i$.
We get $f_{\min} = f^{smo}_3 = -1$ and the minimizer $(1,1,1)$.
It took around 0.68 second.

\subsection{Some explicit examples}

\begin{example}
\label{ex:nonconvex_hol}
Consider the following quadratic SDP arising from \cite{Hol04}:
\be\label{eq:holvariation_opt}
\left\{
\begin{array}{cl}
	\displaystyle\min_{x\in \re^4} & \underbrace{-x_2+(x_1-0.4)^2}_{f_1} +
	\underbrace{2x_1x_3+x_3^2}_{f_2} + \underbrace{x_4+x_1x_4-x_4^2}_{f_3}\\
	\st & G_i(\xdt{i})\succeq 0, \ i=1\ddd 3,
\end{array}\right.
\ee
In the above, $\Dt_1 = \{1,2\}$, $\Dt_2 = \{1,3\}$, $\Dt_3 = \{1,4\}$,
and for each $i$,
\[
G_i(\xdt{i}) \, = \, \left[\begin{array} {ccc}
	2+3x_1^2-x_{i+1} & 2-3x_1 & \\
	2-3x_1 & 1-x_1(x_1+1)-x_{i+1}&\\
	& & \diag(g_i)
\end{array}\right],
\]
where $g_i\coloneqq [(x_1-0.4)^2+(x_{i+1}-0.2)^2-0.5, 1-x_1^2, 1-x_{i+1}^2]^T$.
For (\ref{eq:holvariation_opt}), we solve the sparse matrix Moment-SOS relaxation
(\ref{ksos:MPOP})-(\ref{kmom:MPOP}) with relaxation order $k=2$.
At the optimal solution $y^*$ of (\ref{kmom:MPOP}), the eigenvalues of $M^{(1)}_{\Dt_i}[y^*_{\Dt_i}]$ are:
 \[\begin{aligned}
 M^{(1)}_{\Dt_1}[y^*_{\Dt_1}]:& \quad 1.760,\quad 7.2916 \cdot 10^{-8},\quad 2.6887 \cdot 10^{-8};\\
 M^{(1)}_{\Dt_2}[y^*_{\Dt_2}]:& \quad 2.200,\quad 3.2490 \cdot 10^{-8},\quad 2.2737 \cdot 10^{-8};\\
 M^{(1)}_{\Dt_3}[y^*_{\Dt_3}]:& \quad 2.600,\quad 3.2150 \cdot 10^{-8},\quad 9.4321 \cdot 10^{-9} .
 \end{aligned}
 \]
Note 
$M^{(0)}_{\Dt_i}[y_{\Dt_i}]$ is the $1 \times 1$ matrix of one for all $i$.
Since $d_i = 1$, the rank condition (\ref{eq:ft}) holds with $r_i = 1$ and $t=1$.
By Algorithm~\ref{alg:getminimizer}, we get $f_{\min} = f_2^{smo} = -2.8347$ and the minimizer
\[ (0.7746,   -0.3997,    -0.7746,   -1.0000). \]
Moreover, the condition (\ref{eq:p+fmin_in_idl}) holds with
\[\begin{aligned}
	p_1(x_1,x_2) & = 0.1234+0.4632x_1-1.5067x_1^2+0.1137x_1^3-0.4748x_1^4,\\
	p_2(x_1,x_3) & = 0.2968-0.3535x_1+0.5766x_1^2+0.3064x_1^3+0.2463x_1^4,\\
	p_3(x_1,x_4) & = 2.4145-0.1097x_1+0.9301x_1^2-0.4200x_1^3+0.2285x_1^4.
\end{aligned}\]
It took around 0.34 second.
\end{example}

\begin{example}\label{ex:from_henrion}
Consider the matrix polynomial optimization arising from \cite{henrion2006convergent}:
\be\label{eq:henrionvariation_opt}
\left\{
\begin{array}{cl}
	\displaystyle\min_{x\in \re^3} & \underbrace{(-x_1^2-x_2^2)}_{f_1} +
	\underbrace{(-x_2^2-x_3)}_{f_2}\\
	\st & G_i(\xdt{i})\succeq 0,\ i=1,2.
\end{array}\right.
\ee
In the above, $\Dt_1 = \{1,2\}$, $\Dt_2 = \{1,3\}$, and
\[\begin{aligned}
	G_1(x_1,x_2)\,\coloneqq\,\left[\begin{array}{cc}
		1-4x_1^2x_2^2 & x_1 \\ x_1 & 4-x_1^2-x_2^2
	\end{array}
	\right],\\
	G_2(x_2,x_3)\,\coloneqq\,\left[\begin{array}{cc}
		1-4x_2^2x_3^2 & x_3 \\ x_3 & 4-x_2^2-x_3^2
	\end{array}
	\right].
\end{aligned}\]
We solve the sparse matrix Moment-SOS relaxation (\ref{ksos:MPOP})-(\ref{kmom:MPOP}) with relaxation order $k=5$.
At the optimal solution $y^*$ of (\ref{kmom:MPOP}), the largest five eigenvalues of $M^{(4)}_{\Dt_i}[y^*_{\Dt_i}]$ are:
 \[\begin{aligned}
 M^{(4)}_{\Dt_1}[y^*_{\Dt_1}]:& \quad 264.83,\quad 66.48,\quad 0.9137,\quad 0.2302,\quad 2.2319 \cdot 10^{-6};\\
 M^{(4)}_{\Dt_2}[y^*_{\Dt_2}]:& \quad 265.06,\quad 67.39,\quad 4.1752 \cdot 10^{-6},\quad 2.4326 \cdot 10^{-6},\quad 1.1553 \cdot 10^{-6}.
 \end{aligned}
 \]
Moreover, the the largest five eigenvalues of $M^{(2)}_{\Dt_i}[y^*_{\Dt_i}]$ are:
 \[\begin{aligned}
 M^{(2)}_{\Dt_1}[y^*_{\Dt_1}]:& \quad 16.75,\quad 3.97,\quad 0.0545,\quad 0.0137,\quad 7.7310 \cdot 10^{-7};\\
 M^{(2)}_{\Dt_2}[y^*_{\Dt_2}]:& \quad 16.76,\quad 67.39,\quad 8.1288 \cdot 10^{-7},\quad 3.1082 \cdot 10^{-7},\quad 2.9422 \cdot 10^{-7}.
 \end{aligned}
 \]
Since $d_1 = d_2 = 2$,
the rank condition (\ref{eq:ft}) holds with $r_1 = 4$ and $r_2 = 2$ for $t=4$.
By Algorithm~\ref{alg:getminimizer}, we get $f_{\min} = f_5^{smo} = -8.0683$ and four minimizers:
\[ \begin{gathered}
( 0.1172,   1.9922,   0.1172),\quad ( 0.1172,   -1.9922,  0.1172),\\
( -0.1172,   1.9922,  0.1172),\quad( -0.1172,   -1.9922,   0.1172).
\end{gathered}\]
Moreover, the condition (\ref{eq:p+fmin_in_idl}) holds with
\[
\begin{aligned}
p_1(x_1,x_2) & = 4.2128-0.1557x_2^2-0.2786x_2^4+0.0566x_2^6,
\end{aligned}
\]
and $p_2 = -f_{\min} - p_1$.
It took around 0.56 second.
\end{example}

In the following examples, for neatness of the paper, we do not display the eigenvalues of moment matrices and polynomials
$p_1\ddd p_m$ satisfying (\ref{eq:p+fmin_in_idl})
when the sparse hierarchy (\ref{ksos:MPOP})-(\ref{kmom:MPOP}) is tight.

\begin{example}
Let $\Dt_1 = \{1,2,3\}$ and $\Dt_2 = \{2,3,4\}$.
Consider the optimization
\be\label{eq:sos_convex_opt}
\left\{
\begin{array}{cl}
	\displaystyle\min_{x\in \re^n} & f_1(\xdt{1})+f_2(\xdt{2})\\
	\st & G_i(\xdt{i})\succeq 0, \ i=1,2.
\end{array}\right.
\ee
In the above, $f_1(\xdt{1}) \,\coloneqq\, x_1^6+x_2^6+x_3^6+x_1^2x_2^4+x_2^2x_3^4+x_3^2x_1^4$,
\[
f_2(\xdt{2}) \,\coloneqq\, x_2(x_2^3-1)+x_3(x_3^3-1)+x_4(x_4^3-1)+2x_2^2x_3^2+2x_3^2x_4^2,\]
and for each $i$,
\[
G_i(\xdt{i}) \,\coloneqq\, \left[\begin{array} {rrrrr}
	2-x_{i}^2-2x_{i+2}^2 & \quad & 1+x_{i}x_{i+1} &\quad & x_{i}x_{i+2}\\
	1+x_{i}x_{i+1} & \quad& 2-x_{i+1}^2-2x_{i}^2 & \quad& 1+x_{i+1}x_{i+2}\\
	x_{i}x_{i+2} &\quad & 1+x_{i+1}x_{i+2} &\quad & 2-x_{i+2}^2-2x_{i+1}^2\\
\end{array}\right].
\]
We remark that $f_1$ and $f_2$ are both SOS-convex.
Note that 
\[
f_1(\xdt{1}) = \frac{1}{2}[p(x_1, x_2) + p(x_2, x_3) + p(x_3, x_1)] 
\]
where
$
p(x_1, x_2) = x_1^6+x_2^6 + 2x_1^2x_2^4.
$
The polynomial $p(x_1, x_2)$ is SOS-convex, since
\begin{align*}
	\nabla^2 p(x_1, x_2) &= \begin{bmatrix}
		30x_1^4 + 4x_2^4 & 16x_1x_2^3 \\
		16x_1x_2^3 & 30x_2^4 + 24x_1^2x_2^2
	\end{bmatrix} \\
	&= \bbm 
	14x_1^4 + 4(x_2^2 - 2x_1^2)^2 & 0 \\
	0 & 14x_2^4 + 24x_1^2x_2^2 
	\ebm 
	+ 16 \bbm x_1x_2 \\ x_2^2 \ebm \bbm x_1x_2 \\ x_2^2 \ebm ^T.
\end{align*}
Moreover, note that 
\[
f_2(\xdt{2}) = q(x_2, x_3) + q(x_4, x_3) - x_2 - x_3 - x_4,
\]
where
$
q(x_2, x_3) = x_2^4 + x_3^4/2 + 2x_2^2x_3^2.
$
The polynomial $q(x_2, x_3)$ is SOS-convex, since
\begin{align*}
	\nabla^2 q(x_2, x_3) &= \begin{bmatrix}
		12x_2^2 + 4x_3^2 & 8x_2x_3 \\
		8x_2x_3 & 6x_3^2 + 4x_2^2
	\end{bmatrix} 
	= 
	\bbm
	4x_3^2 & 0 \\
	0 & 4x_2^2 + \frac{2}{3}x_3^2 
	\ebm
	+ 12\bbm x_2 \\ \frac{2}{3}x_3 \ebm \bbm x_2 \\ \frac{2}{3}x_3 \ebm^T.
\end{align*}
Each $-G_i(\xdt{i})$ is SOS-convex; see \cite[Example~10.5.3]{nie2023moment}.
For (\ref{eq:sos_convex_opt}), we solve the sparse matrix Moment-SOS relaxation (\ref{ksos:MPOP})-(\ref{kmom:MPOP}) with relaxation order $k=3$.
By Theorem~\ref{tm:sosconvex}, 
we get $f_{\min} = f_3^{smo} = -1.0342$ and the minimizer is
\[ (0.0000,    0.4421,    0.2586,    0.5207). \]
It took around 0.55 second.

\end{example}

\subsection{Joint minimizers}

Given polynomials $f_1(\xdt{1})$,\ldots, $f_m(\xdt{m})$,
we look for a joint local minimizer $u$ for them, i.e.,
each subvector $\udt{i}$ is a local minimizer of $f_i$.
Consider the unconstrained optimization problem
\be\label{joint_f}
\min\limits_{ \xdt{i} \in \re^{ \Dt_{i} } }  \quad  f_i(\xdt{i}).
\ee
The first and second order optimality conditions are
\be \label{gf=0:Hes>=0}
\nabla_{\xdt{i}} f_i(\udt{i}) = 0, \quad   \nabla_{\xdt{i}}^2  f_i(\udt{i}) \succeq 0.
\ee
The above is necessary for $\udt{i}$ to be a local minimizer for \reff{joint_f}.
If the symbol $\succeq$ in \reff{gf=0:Hes>=0} is replaced by $\succ$, then $\udt{i}$ must be a local minimizer.
As shown in \cite{nie2015loc}, when $f_i(\xdt{i})$ has generic coefficients, if $\udt{i}$ is a local minimizer, then \reff{gf=0:Hes>=0} holds with $ \nabla_{\xdt{i}}^2  f_i(\udt{i}) \succ 0$.

It is interesting to observe that \reff{gf=0:Hes>=0} is equivalent to
\[
\left[\begin{array}{cc}
0     & \nabla_{\xdt{i}} f_i(\udt{i})^T   \\
\nabla_{\xdt{i}} f_i(\udt{i}) &  \nabla_{\xdt{i}}^2  f_i(\udt{i})
\end{array} \right] \succeq 0.
\]
This leads to the sparse optimization problem
\be \label{PMI:joint:Hess>=0}
\left\{
\begin{array}{cl}
\displaystyle\min_{x\in \re^n} & f_1(\xdt{1})+ \cdots + f_m(\xdt{m})\\
\st & \left[\begin{array}{cc}
	0   &  \nabla_{\xdt{i}} f_i(\xdt{i})^T   \\
	\nabla_{\xdt{i}} f_i(\xdt{i}) &  \nabla_{\xdt{i}}^2  f_i(\xdt{i})
\end{array} \right] \succeq 0, i = 1,\ldots, m.
\end{array}\right.
\ee
We remark that if $u$ is a minimizer of \reff{PMI:joint:Hess>=0}
and each $\nabla_{\xdt{i}}^2  f_i(\udt{i}) \succ 0$, then $u$ is a joint local minimizer
for the polynomials $f_i(\xdt{i})$.

\begin{example}\label{eg_joint}
Let $\Dt_1 = \{1,2,3\}$, $\Dt_2 = \{3,4,5\}$, $\Dt_3 = \{5,6,7\}$, and
\[\begin{gathered}
	f_1 = x_1^4 + x_2^4 + x_3^3 - \frac{1}{8}(2x_1x_2+x_3^2+x_3), \quad f_2  =  x_3^4 + x_4^4+x_5^4 -x_3x_4x_5,\\
	f_3  = x_5^3 + x_6^4 + x_7^4 - \frac{1}{8}(x_5^2+x_5-2x_6x_7).\\
\end{gathered}\]
To find a joint local minimizer for them,
we consider the matrix polynomial optimization (\ref{PMI:joint:Hess>=0}) and solve (\ref{kmom:MPOP}).
For the relaxation order $k=3$, we get a minimizer $y^*$ and the flat truncation condition (\ref{eq:ft}) holds with $t=2$.
By Algorithm~\ref{alg:getminimizer},
we get $f^{smo}_4 = -0.0703$ and four minimizers:
\begin{align*}
	x^{(1)} &= (-0.2500, -0.2500, 0.2500, 0.2500, 0.2500, -0.2500, 0.2500), \\
	x^{(2)} &= (-0.2500, -0.2500, 0.2500, 0.2500, 0.2500, 0.2500, -0.2500), \\
	x^{(3)} &= (0.2500, 0.2500, 0.2500, 0.2500, 0.2500, -0.2500, 0.2500),\\
	x^{(4)} &= (0.2500, 0.2500, 0.2500, 0.2500, 0.2500, 0.2500, -0.2500).
\end{align*}
It took around 0.27 second.
Moreover, one may check that for every $i=1\ddd 3$ and $j=1\ddd 4$, it holds
$\nabla^2 f_i(\udt{i}^{(j)})\succ 0.$
Therefore, all of $u^{(1)}\ddd u^{(4)}$ are joint local minimizers of $f_1,f_2$ and $f_3$.
\end{example}

\begin{example}
In Example~\ref{eg_joint}, if we change $f_2$ to
\[
f_2  =  x_3^4 + x_4^4+x_5^4,
\]
then the sparse matrix moment relaxation (\ref{kmom:MPOP})
is infeasible for the relaxation order $k=4$.
This means (\ref{PMI:joint:Hess>=0}) is infeasible and there do not exist joint minimizers.
Therefore, we consider the regularized optimization problem
\be \label{joint:purturb}
\left\{
\begin{array}{cl}
	\displaystyle\min_{x\in \re^7, z\in\re^3} &   z_1 + z_2 + z_3  \\
	\st & \left[\begin{array}{cc}
		z_i & \nabla f_i(\xdt{i})^T   \\
		\nabla f_i(\xdt{i}) &  z_i I_{n_i}+\nabla^2  f_i(\xdt{i})
	\end{array} \right] \succeq 0, \ i = 1,2,3.
\end{array}\right.
\ee
For each $i$, let $\hat{x}_{\Dt_i}\coloneqq (\xdt{i},z_i)$.
Then (\ref{joint:purturb}) is a new sparse matrix polynomial optimization problem.
We solve the sparse matrix Moment-SOS relaxation (\ref{ksos:MPOP})-(\ref{kmom:MPOP}) with $k=4$,
and get a lower bound $0.0017$ for the minimum value of (\ref{joint:purturb}).
By Algorithm~\ref{alg:getminimizer}, the flat truncation condition (\ref{eq:ft})
holds with $t=1$, and we get four minimizers $(x^{(j)},z^{(j)})\, j=1\ddd 4$ for (\ref{PMI:joint:Hess>=0}), which are:
\[
\begin{gathered}
	x^{(1)} = (-0.2500,-0.2500,0.2260,0.0000,0.2260, -0.2500, 0.2500),\\
	x^{(2)} = (-0.2500,-0.2500,0.2260,0.0000,0.2260, 0.2500, -0.2500),\\
	x^{(3)} = (0.2500,0.2500,0.2260,0.0000,0.2260, -0.2500, 0.2500),\\
	x^{(4)} = (0.2500,0.2500,0.2260,0.0000,0.2260, 0.2500, -0.2500),\\
	z^{(1)} =z^{(2)} =z^{(3)} =z^{(4)} = (0.0007,0.0069,0.0007).\\
\end{gathered}
\]
Furthermore, for all $j=1\ddd 4$, we have
\[ \Vert\nabla f_1(\xdt{1}^{(j)})\Vert = 0.0283,\ \Vert\nabla f_2(\xdt{2}^{(j)})\Vert = 0.0653,\ \Vert\nabla f_3(\xdt{3}^{(j)})\Vert = 0.0283,  \]
\[ \nabla^2 f_1(\xdt{1}^{(j)}) \succ 0 ,\ \nabla^2 f_2(\xdt{2}^{(j)}) \succ 0,\ \nabla^2 f_3(\xdt{3}^{(j)}) \succ 0.  \]
It took around 6.42 seconds.
\end{example}

\subsection{Center points for sets given by PMIs}

Let $G_1\ddd G_m$ be given matrix polynomials in $z\in\re^{n}$.
For each $i$, consider the semialgebraic set
\be\label{eq:Pi}
P_i \,\coloneqq\, \{ z\in \re^{n} : G_i(z) \succeq 0 \}.
\ee
The sets $P_1 \ddd P_m$ may or may not intersect.
We look for a point $v \in\re^{n}$
such that the sum of squared distances from $v$ to all $P_i$ is minimum.
This can be formulated as the optimiation problem
\be\label{eq:spec_projection}
\left\{
\begin{array}{cl}
\displaystyle\min_{z_1, \ldots, z_m, v \in\re^{n} } &  \displaystyle \sum_{i=1}^m \| z^{(i)} - v  \|^2 \\
\st & G_i(z^{(i)})\succeq 0, \ i=1\ddd m.
\end{array}\right.
\ee
Let $x\coloneqq (z^{(1)}\ddd z^{(m)}, v)$, and denote
\[
\xdt{i}\coloneqq (z^{(i)},v), \quad f_i(\xdt{i})\coloneqq \Vert z^{(i)} - v\Vert^2, \quad i=1\ddd m.
\]
Then, (\ref{eq:spec_projection}) is a sparse matrix polynomial optimization problem in the form of (\ref{spar:matPOP}).
For every minimizer $x^* = (z^{(1,*)}\ddd z^{(m,*)},v^*)$ of \reff{eq:spec_projection}, the point $z^{(i,*)}$ is the projection of $v^*$ to $P_i$.
Note that if $P_1 \cap \cdots \cap P_m \ne \emptyset$, then the minimum value of \reff{eq:spec_projection} is $0$.

\begin{example}
\label{ex:projection}
Consider the matrix polynomial
\[
F(z) =
\left[
\begin{array}{rrr}
	z_1^2 + z_3^2 & -z_1z_2 & -z_1z_3 \\
	-z_1z_2 & z_2^2+z_1^2 & -z_2z_3 \\
	-z_1z_3& -z_2z_3 &  z_3^2+z_2^2
\end{array}
\right].
\]
Let $G_i(z) \, \coloneqq \, I_3 - F(z-c_i)$, where
\[
c_1 = (2, 0, 0), \quad c_2 = (0, 2, 0), \quad c_3 = (0, 0, 2).
\]
The matrix polynomial $F(z)$ is SOS-convex (see Example~\ref{soscon}),
thus the sparse moment relaxation \reff{kmom:MPOP} is tight for all relaxation orders.
We solve (\ref{kmom:MPOP}) for $k=1$ and get $f_{\min} = f^{smo}_1 = 1.4291$.
Moreover, we find a minimizer of (\ref{eq:spec_projection}), which gives the center point $v^*$ and its projections $z^{(1,*)}, z^{(2,*)}, z^{(3,*)}$:
\[
\begin{aligned}
	v^*= (0.8591,    0.8591,    0.8591), \quad &z^{(1,*)} = (1.4226,    0.5774,    0.5774), \\
	z^{(2,*)}= (0.5774,    1.4226,    0.5774), \quad  &z^{(3,*)} = (0.5774,    0.5774,    1.4226).
\end{aligned}
\]
It took around 0.16 second.
\end{example}

\begin{example}
Consider the matrix polynomials
\[
G_1 = \bbm \frac{z_1}{2} & z_1^2+1 \\
z_1^2+1 & \frac{z_2}{2}
\ebm, \quad
G_2 = \bbm \frac{z_2}{2} & z_2^2+1 \\
z_2^2+1 & \frac{z_3}{2}
\ebm, \quad
G_3 = \bbm \frac{z_1}{2} & z_3^2+1 \\
z_3^2+1 & \frac{z_3}{2}
\ebm.
\]
Any two of $P_1, P_2, P_3$ intersect, but $P_1 \cap P_2 \cap P_3 = \emptyset$.
This is because if all $G_i(z) \succeq 0$, then it holds
\[A =
\bbm z_1 & z_1^2+1 & z_3^2+1\\
z_1^2+1 & z_2 & z_2^2+1 \\
z_3^2+1 & z_2^2+1 & z_3
\ebm \succeq 0.
\]
However, there is no $z$ satisfying the above.
By Algorithm~\ref{alg:getminimizer}, we get $f_{\min} = f^{smo}_1 = 206.3980$ and get a minimizer of (\ref{eq:spec_projection}), which gives the center point $v^*$ and its projections $z^{(1,*)}, z^{(2,*)}, z^{(3,*)}$:
\[
\begin{aligned}
	v^* = (6.4613,    6.4613,    6.4613),\quad & z^{(1,*)} = (0.5960,   12.3262,    6.4615),\\
	z^{(2,*)} = (6.4615,    0.5960,   12.3262),\quad & z^{(3,*)} = (12.3262,    6.4615,    0.5960).
\end{aligned}
\]
It took around 0.18 second.
\end{example}

\subsection{Multisystem static $\mc{H}_2$ controller synthesis} \label{sc:control}
The static $\mc{H}_2$ controller synthesis \cite{Hol04,SchererWeiland2000} is a classical problem in optimal control that has wide applications in power systems, robotics, industrial process control, etc.
For $i=1\ddd m$, 
let $A_i\in \re^{a_i\times a_i}$,
$B_{i}\in\re^{a_i\times b_i}$,
$C_{i}\in\re^{a_i\times p}$,
$D_{i}\in\re^{c_i \times a_i}$, and
$E_{i}\in\re^{q\times a_i}$.
Let $K = (K_{jk}) \in \re^{p\times q}$ be the static controller matrix,
and let each 
$
X_i = \big((X_{i})_{jk}\big) \in \mc{S}^{a_i}
$
be the Lyapunov matrix for the $i$th system.
The multisystem static $\mc{H}_2$ controller synthesis problem is to find a controller matrix $K^*$ and Lyapunov matrices $X_1^*\ddd X_m^* \succeq 0$ that solve the following optimization problem
\be\label{eq:H2_controller} 
\left\{
\begin{array}{cl}
	\displaystyle\min_{K,X_1\ddd X_m} & \tr{D_{1}X_1D_{1}^T} + \cdots + \tr{D_{m}X_mD_{m}^T}\\
	\st & {\mc{A}_i}(K)X_i + X_i{\mc{A}_i}(K)^T + B_{i}B_{i}^T = 0, \\
	& X_i\succeq 0,\, i=1\ddd m, \\
	& \xi I_p-KK^T\succeq 0.
	\end{array}\right.
\ee
In the above, each ${\mc{A}_i}(K)$ is the closed-loop matrix for the $i$th system:
\[ {\mc{A}_i}(K)\coloneqq A_i+C_{i}KE_{i}, \]
the equality constraints are called Lyapunov equations,
and the matrix inequality constraint $\xi I_p-{KK^T}\succeq 0$ gives an upper bound for the spectral norm of $K$ that caps its peak amplification in every input direction.
It is clear that (\ref{eq:H2_controller}) is nonconvex.
Denote by $z_0, z_i$ the vectorization of $K$ and the upper-triangular part of $X_i$ respectively, i.e.,
\[ 
z_0\coloneqq (K_{11},K_{12},\ldots K_{pq}),\quad
z_i = \big((X_{i})_{11}, (X_{i})_{12} \ddd (X_{i})_{22},\ldots, (X_{i})_{a_ia_i}\big). 
\]
Let $x \coloneqq (z_0,z_1\ddd z_m)$, and let $\xdt{i}\coloneqq (z_0,z_i).$
Then $x \in \re^{n}$ and $\xdt{i} \in \re^{n_i}$, with 
\[ n = pq+ a_1(a_1 +1)/2 + \cdots + a_m(a_m +1)/2,\quad n_i = pq+a_i (a_i +1)/2. \]
Problem (\ref{eq:H2_controller}) becomes a sparse matrix polynomial optimization problem in the form of (\ref{spar:matPOP}).

\begin{example}
Consider the multisystem static $\mc{H}_2$ controller synthesis problem (\ref{eq:H2_controller}) with $m = 4$.
We set $\xi = 10$ and 
\[
A_1 = \left[\begin{array}{rr} -2 & 2 \\ 2 & 1 \end{array}\right],\ 
A_2 = \left[\begin{array}{rr} 1 & -3 \\ 1 & -2 \end{array}\right],\ 
A_3 = \left[\begin{array}{rr} 1 & -1 \\ 3 & -2 \end{array}\right],\ 
A_4 = \left[\begin{array}{rr} 0 & -1 \\ 2 & -2 \end{array}\right],
\]
\[
B_{1} = \left[\begin{array}{rr} -4 & 2 \\ -3 & 2 \end{array}\right],\ 
B_{2} = \left[\begin{array}{rrr} -3 & -4 & -1 \\ -2 & -1 & 0 \end{array}\right],\ 
B_{3} = \left[\begin{array}{rrr} 0 & 3 & -1 \\ 1 & 0 & 2 \end{array}\right],\ 
B_{4} = \left[\begin{array}{rrr} -1 & -2 & 1 \\ -1 & 1 & -2 \end{array}\right],
\]
\[
C_{1} = \left[\begin{array}{rr} 2 & -1 \\ 0 & 1 \end{array}\right],\ 
C_{2} = \left[\begin{array}{rr} 2 & -1 \\ -2 & 3 \end{array}\right],\ 
C_{3} = \left[\begin{array}{rr} 0 & 1 \\ 1 & 0 \end{array}\right],\ 
C_{4} = \left[\begin{array}{rr} 0 & 1 \\ 1 & -1 \end{array}\right],
\]
\[
D_{1} = \left[\begin{array}{rr} 2 & -2 \\ 2 & 2 \\ -1 & -1 \end{array}\right],\ 
D_{2} = \left[\begin{array}{rr} 2 & -1 \\ 3 & -1 \end{array}\right],\ 
D_{3} = \left[\begin{array}{rr} 1 & -3 \\ -1 & 2 \\ 3 & 3 \end{array}\right],\ 
D_{4} = \left[\begin{array}{rr} 2 & -2 \\ 0 & 1 \\ -2 & -2 \end{array}\right],
\]
\[
E_{1} = \left[\begin{array}{rr} 1 & 3 \\ 2 & 2 \end{array}\right],\ 
E_{2} = \left[\begin{array}{rr} -1 & 1 \\ 2 & 0 \end{array}\right],\ 
E_{3} = \left[\begin{array}{rr} 3 & -1 \\ 1 & 4 \end{array}\right],\ 
E_{4} = \left[\begin{array}{rr} 0 & -1 \\ 3 & 0 \end{array}\right].
\]
By Algorithm~\ref{alg:getminimizer}, we get $f_{\min} = f^{smo}_2 = 81.0282$ for $k=2$. The optimal controller $K^*$ and Lyapunov matrices $X_{1}^*\ddd X_{4}^*$ are
\[
\begin{gathered}
	K^* = \left[\begin{array}{rr}    -0.3826 &  -1.5343 \\   -1.3829 &  -0.7662 \end{array}\right],\quad
	X_1^* = \left[\begin{array}{rr}  1.5084  &  1.2904 \\   1.2904  &  1.1362 \end{array}\right], \quad
	X_2^* = \left[\begin{array}{rr}  2.0957  &  1.7426 \\   1.7426  &  2.3817 \end{array}\right], \\
	X_3^* = \left[\begin{array}{rr}  1.4175  &  -0.2049 \\   -0.2049  &  0.3140 \end{array}\right], \quad
	X_4^* = \left[\begin{array}{rr}  1.2119  &  -0.5593 \\   -0.5593  &  1.0566 \end{array}\right].      
\end{gathered}
\]
It took around 1.33 seconds.
In comparison, we also apply the dense relaxation (\ref{dense_mom}) to solve this problem.
The optimal value of (\ref{dense_mom}) equals $0$ when $k=1$, which is not tight.
When $k=2$, the dense relaxation took 150.25 seconds to get the global minimum and minimizers of (\ref{eq:H2_controller}), which are the same as those returned by Algorithm~\ref{alg:getminimizer}.
\end{example}

\subsection{Some random matrix optimization problems}

\begin{example}\label{ex:quad_random_sosconvex}
Consider the matrix polynomial optimization problem
\be\label{eq:qcquadsdp}
\left\{
\begin{array}{cl}
	\displaystyle\min_{x\in \re^n} & \sum\limits_{i=1}^m\underbrace{\big({\xdt{i}^{[2]}\big)}^T D_i\xdt{i}^{[2]} + \xdt{i}^T Q_i\xdt{i} + p_i^T\xdt{i}}_{f_i}\\
	\st & G_i(\xdt{i})\succeq 0, \ i=1\ddd m.
\end{array}\right.
\ee
In the above, each set $\Delta_i$ is selected as
\be\label{eq:Dt_i}
\Dt_i \,\coloneqq\, \{ j \in [n] : 1 \le j-(\omega-1)(i-1) \le \omega \},
\ee
and
\[
x_{\Dt_i}^{[2]}  \,  \coloneqq  \,  (x_j^2)_{j\in \Dt_i} .
\]
The cardinality of each $\Dt_i$ is $\omega$, and $m,n$
are integers such that $(\omega-1)m+1 = n$.
We randomly generate $D_i\coloneqq \hat{D}^T\hat{D}$
with $\hat{D} = {\tt rand}(\omega)$ in {\tt MATLAB}.
So, $D_i$ is psd and has only nonnegative entries.
We also randomly generate $Q_i \coloneqq \hat{Q}^T\hat{Q}$ with $\hat{Q} = {\tt randn}(\omega)$ and $p_i \coloneqq {\tt randn}(\omega,1)$ in {\tt MATLAB}.
So $Q_i$ is psd but may have negative entries.
Thus, each $f_i$ is SOS-convex; see \cite[Example~7.1.4]{nie2023moment}.
Moreover, we let $G_i$ be the $\ell$-by-$\ell$ matrix polynomial
randomly generated as
\be\label{eq:quadratic_Gi}
G_i(\xdt{i}) \coloneqq C_i +\sum_{s \in \Dt_i} B_{i,s}x_{s} -
(x_{\Dt_i} \otimes I_{\ell})^T  A_i (x_{\Dt_i} \otimes I_{\ell}),
\ee
where each $C_i\in\mc{S}^{\ell}_+$ and $A_i\in\mc{S}_+^{\ell\omega}$
are randomly generated in the same way as for $Q_i$, and each
$B_{i,s} = \hat{B} + \hat{B}^T$ with $\hat{B} = {\tt randn}(\ell)$ in {\tt MATLAB}.
For such choices, each set $K_{\Dt_i}$ is nonempty
(it contains the origin) and each $G_i(\xdt{i})$
is SOS-concave (see the case (iii) on the bottom of page~404 of \cite{nie2011PMI}).
By Theorem~\ref{tm:sosconvex}, we have $f^{smo}_k = f_{\min}$ for all $k\ge 2$.

\bnum

\item[(i)] We consider the values $m=5,10,15,20$, $\omega = 5, 10$ and $\ell = 5, 10$.
For each case of $(\omega, \ell, m)$, we generate $10$
random instances and solve the respective sparse moment relaxations
(\ref{kmom:MPOP}) for order $k=2$.
The dense moment relaxations \reff{dense_mom} are solved for the same order $k$.
The average computational time (in seconds) is reported in Table~\ref{tab:quadQCSDP}.
The time for solving the sparse relaxation is displayed on the left,
and the time for solving the dense relaxation is displayed on the right.
The text ``{\tt oom}" means that the computer is out of memory for the computation.
\begin{table}
	\caption{ Computational time (in seconds) for solving (\ref{eq:qcquadsdp}) generated in Example~\ref{ex:quad_random_sosconvex} (i)
		by the sparse moment relaxation \reff{kmom:MPOP}, shown on the left,
		and by the dense moment relaxation \reff{dense_mom}, shown on the right.
		The text ``{\tt oom}" means the computer is out of memory.  }
	\label{tab:quadQCSDP}
	\begin{tabular}{|l|c|c|c|c|}\hline
		& $m=5$ & $m=10$ & $m=15$ & $m=20$  \\ \hline
		$(\omega, \ell) = (5,5)$   & (0.65, 1879.96) & (0.99, oom)  & (1.06, oom) & (1.20, oom) \\ \hline
		$(\omega, \ell) = (5,10)$  & (2.08, oom) &  (5.97, oom) & (9.02, oom) & (12.21, oom) \\ \hline
		$(\omega, \ell) = (10,5)$  & (0.99, oom) & (26.61, oom)  & (30.47, oom) & (37.22, oom) \\ \hline
		$(\omega, \ell) = (10,10)$ & (73.45, oom) & (169.83, oom)  & (212.81, oom) & (617.06, oom) \\ \hline
	\end{tabular}
\end{table}

\item[(ii)]
To demonstrate the scalability of our sparse relaxations, 
we test on the sparse problem \reff{eq:qcquadsdp} for various substantially large values of $n$ and different relaxation orders.
Numerical results are presented in Table~\ref{tab:largescale}. 
For each case of $(\omega, \ell, m)$, we generate $10$ random instances and solve the corresponding sparse moment relaxations for $k=2$ and $k=3$. The notation ``{\tt oom}'' indicates that the computation ran out of memory.
\enum

\begin{table}
	\caption{ Computational time (in seconds) for solving (\ref{eq:qcquadsdp}) generated in Example~\ref{ex:quad_random_sosconvex} (ii)}
	\label{tab:largescale}
	\begin{tabular}{|l|c|c|c|c|}\hline
		& $n$ & $k=2$ & $k=3$  \\ \hline
		$(\omega, \ell, m) = (3,3,100)$ & 201 & 0.86 & 7.61   \\ \hline
		$(\omega, \ell, m) = (3,3,300)$ & 601 & 2.04 & 24.11   \\ \hline
		$(\omega, \ell, m) = (3,3,500)$ & 1001 & 4.16 & oom   \\ \hline
		$(\omega, \ell, m) = (3,3,1000)$ & 2001 & 10.45 & oom   \\ \hline
		$(\omega, \ell, m) = (5,3,50)$ & 201 & 1.59 & 99.65  \\ \hline
		$(\omega, \ell, m) = (5,3,100)$ & 401 & 3.40 & 129.79   \\ \hline
		$(\omega, \ell, m) = (5,3,300)$ & 1201 & 9.43 & oom   \\ \hline
		$(\omega, \ell, m) = (5,5,30)$   & 121 & 2.61 & 272.30   \\ \hline
		$(\omega, \ell, m) = (5,5,50)$   & 201 & 3.42 & 2026.52   \\ \hline
		$(\omega, \ell, m) = (5,5,100)$   & 401 & 8.73 & oom   \\ \hline
		$(\omega, \ell, m) = (5,5,200)$   & 801 & 16.77 & oom   \\ \hline
		$(\omega, \ell, m) = (10,5,30)$   & 271 & 49.60 & oom   \\ \hline
		$(\omega, \ell, m) = (10,5,50)$   & 451 & 94.63 & oom   \\ \hline
		$(\omega, \ell, m) = (10,10,20)$   & 226 & 617.06 & oom   \\ \hline
		$(\omega, \ell, m) = (10,10,30)$   & 271 & oom  & oom  \\ \hline
	\end{tabular}
\end{table}
\end{example}

\begin{example}\label{ex:random_nonconvex}
We still consider the sparse matrix polynomial optimization problem (\ref{eq:qcquadsdp}).
For each $i$, we randomly generate $D_i  \coloneqq \hat{D} + \hat{D}^T$
with $\hat{D} = {\tt randn}(\omega)$ in {\tt MATLAB},
and we randomly generate $Q_i$, $p_i$, $G_i$
in the same way as in Example~\ref{ex:quad_random_sosconvex}.
Then, the generated problem (\ref{eq:qcquadsdp}) is typically nonconvex.
We consider the values
\[
m=5,10,15,20, \qquad \omega = 5, 10, \qquad \ell = 5, 10.
\]
For each case of $(m, \omega, \ell)$, we generate $10$
random instances. 
We solve the corresponding sparse moment relaxations~\reff{kmom:MPOP} for $k=2$ and $k=3$, and apply Algorithm~\ref{alg:getminimizer} to check the tightness of the relaxation and extract minimizers. The computational results are reported in Tables~\ref{tab:quadQCSDP_nonconvexk=2} and \ref{tab:quadQCSDP_nonconvexk=3}. The number of random instances (out of ten) for which \reff{kmom:MPOP} is tight is displayed in parentheses. The average computational time is reported in seconds.
For comparison, we also solve the dense relaxation~\reff{dense_mom} with $k=2$ for each random instance. The dense relaxation~\reff{dense_mom} is only solvable for $m = \omega = \ell = 5$, with an average computational time of $1861.02$ seconds. 
For all other values of $(m, \omega, \ell)$, the dense relaxation~\reff{dense_mom} is not solvable due to memory shortage.
We also note that when increasing the relaxation order to $k=3$, most instances encounter out-of-memory issues, except for the cases $(\omega, \ell) = (5,5)$ and $(\omega, \ell, m) = (10, 5, 5)$.
Interestingly, although the sparse matrix polynomial optimization problems are generally nonconvex, the sparse moment relaxations are tight with $k=2$ in most cases. One possible reason is that the constraints are given by SOS-convex matrix polynomials, and the objective function is the sum of a convex quadratic function and a typically nonconvex quartic form. Investigating the tightness of the sparse Moment-SOS hierarchy for nonconvex matrix polynomial optimization with special structures is an interesting direction for future work.

\begin{table}
	\caption{
		Computational time (in seconds) for solving the nonconvex optimization
		\reff{eq:qcquadsdp} generated in Example~\ref{ex:random_nonconvex}
		by the sparse moment relaxation \reff{kmom:MPOP} with $k=2$.
		The number of instances for which \reff{kmom:MPOP} is tight
		is shown inside the parenthesis. }
	\label{tab:quadQCSDP_nonconvexk=2}
	\begin{tabular}{|l|c|c|c|c|}\hline
		$k = 2$ & $m=5$ & $m=10$ & $m=15$ & $m=20$  \\ \hline
		$(\omega, \ell) = (5, 5)$ & 2.49 (10) & 4.77 (10)  & 7.20 (9) & 9.68 (9) \\ \hline
		$(\omega, \ell) = (5, 10)$ & 5.13 (10) &  10.86 (10) & 17.28 (10) & 23.18 (10) \\ \hline
		$(\omega, \ell) = (10, 5)$ & 45.97 (10) & 95.29 (9)  & 148.81 (8) & 204.75 (8) \\ \hline
		$(\omega, \ell) = (10,10)$ & 113.42 (10) & 245.34 (10)  & 406.90 (10) & 611.10 (10) \\ \hline
	\end{tabular}
\end{table}
\begin{table}
	\caption{
		Computational time (in seconds) for solving the nonconvex optimization
		\reff{eq:qcquadsdp} generated in Example~\ref{ex:random_nonconvex}
		by the sparse moment relaxation \reff{kmom:MPOP} with $k=3$.
		The number of instances for which \reff{kmom:MPOP} is tight
		is shown inside the parenthesis. }
	\label{tab:quadQCSDP_nonconvexk=3}
	\begin{tabular}{|l|c|c|c|c|}\hline
		$k = 3$ & $m=5$ & $m=10$ & $m=15$ & $m=20$  \\ \hline
		$(\omega, \ell) = (5, 5)$ & 38.05 (10) & 84.49 (10)  & 137.06 (9) & 192.77 (10) \\ \hline
		$(\omega, \ell) = (5, 10)$ & oom &  oom & oom & oom \\ \hline
		$(\omega, \ell) = (10, 5)$ & 947.41 (10) & oom & oom & oom \\ \hline
		$(\omega, \ell) = (10,10)$ & oom & oom  & oom & oom \\ \hline
	\end{tabular}
\end{table}
\end{example}

\section{Conclusions}
\label{sc:con}

This paper studies the sparse polynomial optimization problem with matrix constraints,
given in the form \reff{spar:matPOP}. We study the sparse matrix Moment-SOS hierarchy of
\reff{ksos:MPOP}-\reff{kmom:MPOP} to solve it.
First, we prove a sufficient and necessary condition for this sparse hierarchy to be tight.
This is the condition~\reff{eq:p+fmin_in_idl} shown in Theorem~\ref{thm:sepa:match}.
We also discuss how to detect the tightness and how to extract minimizers.
The main criterion is to use flat truncation \reff{eq:ft}, which is justified in Theorem~\ref{tm:getminimizer}.
When this optimization problem is convex,
we prove the sufficient and necessary condition
for the tightness holds under some general assumptions.
In particular, when the problem is SOS-convex,
we show that the sparse matrix Moment-SOS relaxation is tight for all relaxation orders.
These results are shown in Theorems~\ref{tm:convex} and \ref{tm:sosconvex}.
Numerical experiments are provided to demonstrate how Algorithm~\ref{alg:getminimizer} can be applied to detect tightness of sparse matrix Moment-SOS relaxations and extract minimizers for \reff{spar:matPOP}.

\bigskip
\noindent
{\bf Acknowledgments}
This project was begun at a SQuaRE at the American Institute of Mathematics (AIM).
The authors would like to thank AIM for providing
a supportive and mathematically rich environment.

\bigskip
\noindent
{\bf Funding}
Jiawang Nie and Linghao Zhang are partially supported by the NSF grant DMS-2110780.
Xindong Tang is partially supported by the Hong Kong Research Grants Council HKBU-15303423.

\end{document}